\documentclass{amsart}
\usepackage{comment,mathrsfs, amssymb,amsmath,url}

\usepackage{hyperref}
\usepackage{tikz}
\usepackage[english]{babel}
\usepackage[utf8]{inputenc}

\usetikzlibrary{arrows}
\usetikzlibrary{petri}
\usetikzlibrary{backgrounds}
\usetikzlibrary{decorations.pathmorphing}
\usetikzlibrary{calc}
\tikzstyle{vertex}=[circle,draw=blue!50,fill=blue!20,thick]
\tikzstyle{diedge}=[->,shorten <=1pt,>=angle 90,semithick]
\tikzstyle{forced}=[->,shorten <=1pt,>=angle 90,semithick,dashed]
\tikzstyle{bigbad}=[line width=3pt]
\tikzstyle{every token}=[draw=blue!50,fill=blue!20,thick]

\def\relstr#1{\mathbb{#1}}
\def\alg#1{\mathbf{#1}}

\let\phi\varphi
\def\zet{{\mathbb Z}}

\newcommand\tupl[1]{\overline{#1}}
\let\tuple\tupl

\DeclareMathOperator\absorbs{\trianglelefteq}

\DeclareMathOperator\Jabsorbs{\trianglelefteq_J}
\let\subset\subseteq
\def\compNP{{\textsf{NP}}}

\def\compP{{\textsf{P}}}

\def\algA{{\mathbf{A}}}

\newcommand\True{\mathrm{True}}

\newcommand\AND{\mathrel{\wedge}}

\def\AND{\mathop{\wedge}}

\def\en{{\mathbb N}}

\let\epsilon\varepsilon

\theoremstyle{plain}
\newtheorem{theorem}{Theorem}
\newtheorem{lemma}[theorem]{Lemma}
\newtheorem{corollary}[theorem]{Corollary}
\newtheorem{proposition}[theorem]{Proposition}

\theoremstyle{definition}
\newtheorem{definition}[theorem]{Definition}
\newtheorem{problem}[theorem]{Problem}
\newtheorem{algorithm}[theorem]{Algorithm}

\begin{document}

\title{Deciding absorption}
\author{Libor Barto}
\address{Department of Algebra, Charles University\\
Sokolovsk\'a 83, 186 00 Praha 8, Czech Republic}
\email{libor.barto@gmail.com}
\urladdr{http://www.karlin.mff.cuni.cz/~barto/}
\author{Alexandr Kazda}
\address{Institute of Science and Technology Austria\\
Am Campus 1, 3400, Klosterneuburg, Austria}
\email{alex.kazda@gmail.com}
\urladdr{http://pub.ist.ac.at/~akazda}
\maketitle

\begin{abstract}
We characterize absorption in finite idempotent algebras by means of J\'onsson absorption and cube term blockers. As an application we show that it is decidable whether a given subset is an absorbing subuniverse of an algebra given by the tables of its basic operations. 
\end{abstract}


\section{Introduction}

A subuniverse $B$ (or a subalgebra $\alg{B}$) of an algebra $\algA$ is \emph{absorbing} if $\algA$ has an idempotent term $t$ such that $t(a_1, \dots, a_n) \in B$ whenever all but at most one of the arguments $a_1, \dots, a_n$ are in $B$ and the exceptional argument is from $A$. 
Absorption played a key role in the proofs of several results concerning
Maltsev conditions and the complexity of the constraint satisfaction problem
(see eg.~\cite{barto-kozik-cyclic-terms-and-csp},~\cite{barto-kozik-stanovsky-absorption-solvability},~\cite{barto-kozik-bw-2014},
and~\cite{barto-kozik-approximation}) and it is directly used in the simplified algorithm for conservative constraint satisfaction problems from~\cite{libor-conservative}.
The applicability of this concept has two sources.
On one hand, some useful properties of algebras and their subpowers are inherited by their absorbing subalgebras (see eg. Lemma~\ref{lemFence}) and, on the other hand, a finite algebra has a proper absorbing subalgebra under mild assumptions (see the Absorption Theorem from~\cite{barto-kozik-cyclic-terms-and-csp}).  

However, from the definition of absorption, it is not clear how to decide,
given finite $\algA$ and its subuniverse $B$, whether $B$ absorbs $\algA$. This problem is closely connected to the problem of deciding whether a given finite algebra has a near unanimity (NU) term of some arity. Indeed, a finite 
algebra $\algA$ has an NU term if and only if every singleton
in $A$ is an absorbing subuniverse of the full idempotent reduct of $\algA$. 

Mikl\'os Mar\'oti has shown~\cite{maroti-nu-is-decidable} that the near unanimity problem is decidable, although his algorithm has enormous time complexity. 
 For idempotent algebras, we provide the following generalization of his result.
\begin{theorem} \label{thm-deciding-absorption}
Deciding, given a finite idempotent $\algA$ and its subuniverse $B$, whether $B\absorbs \algA$ is co-\compNP-complete, and  is fixed parameter tractable when parametrized by the product of the arities of the basic operations in $\algA$.
\end{theorem}
Note that this theorem, as stated, does not imply Mar\'oti's NU result for non-idempotent algebras since it is not clear at first how to decide whether $B$ is an absorbing subuniverse of the full idempotent reduct of a given $\algA$. 
However, using the techniques from~\cite{zhuk-kazda} it is straightforward to adapt our algorithm to this more general problem in exchange for a worse time complexity.

To prove Theorem~\ref{thm-deciding-absorption} we generalize, for finite idempotent algebras, a relatively recent discovery~\cite{BIMMVW,MM08b} that having an NU term is equivalent to the conjunction of two weaker Maltsev conditions -- having J\'onsson terms and having a cube term. 

Having J\'onsson terms is a classical condition which characterizes algebras in congruence distributive varieties~\cite{jonsson}. In a similar way in which absorption generalizes NU terms, J\'onsson terms can be generalized to J\'onsson absorption (see Definition~\ref{def-jonsson-abs}). In particular, ``$B$ J\'onsson absorbs $\algA$'' is a weakening of ``$B$ absorbs $\algA$''. 

Having a cube term is a substantially more recent condition which characterizes, for finite algebras, the property of having few subpowers and number of other important properties~\cite{BIMMVW,KS12}. A useful equivalent condition to ``$\algA$ has a cube term'' (assuming $\algA$ is finite and idempotent) is ``$\algA$ has no cube term blockers'', where a cube term blocker is a pair of subuniverses with certain properties, see~\cite{markovic-maroti-mckenzie-blockers}. In particular, since cube terms are weaker than NU terms, no algebra with an NU term can have a cube term blocker.
This fact can be also generalized to absorption: some of the cube term blockers, which we call $B$-blockers, prevent $B$ from being an absorbing subuniverse of $\algA$ (see Definition~\ref{def-blocker}). In other words, ``$\algA$ has no $B$-blockers'' is a weakening of ``$B$ absorbs $\algA$''.

The main result of this paper shows that absorption is equivalent to the conjunction of the two of its weaker forms described above:

\begin{theorem} \label{MainThm}
Let $\algA$ be a finite idempotent algebra and $B$ a subuniverse of $\algA$. Then the following are equivalent.
\begin{itemize}
\item[(i)] $B$ absorbs $\algA$.
\item[(ii)] $B$ J\'onsson absorbs $\algA$ and $\algA$ has no $B$-blocker.
\end{itemize}
\end{theorem}

\section{Preliminaries}

We use rather standard universal algebraic terminology~\cite{burris,uabook}: We
denote algebras by capital letters $\algA$, $\alg{B}$, $\alg{R}$, \dots in
boldface. The same letters  $A$, $B$, $R$, \dots in the plain font are used to
denote universes of algebras. We will call a subuniverse or a subalgebra of a power a \emph{subpower}. We mostly consider finite algebras that are idempotent, that is, each  basic operation $f$ satisfies the identity $f(x, x, \dots, x) \approx x$. 


%

\subsection{Absorption}

\begin{definition}
Let $\algA$ be an algebra and $B$ a subuniverse of $\algA$.
We say that $B$ \emph{absorbs} $\algA$ if there is an idempotent term $t$ of $\algA$ such that whenever
$b_1,\dots,b_n\in B$ and $a\in A$, then
\begin{align*}
  t(a,b_2,b_3,\dots,b_{n-1},b_n)&\in B\\
  t(b_1,a,b_3,\dots,b_{n-1},b_n)&\in B\\
  \vdots&\\
  t(b_1,b_2,b_3,\dots,b_{n-1},a)&\in B.
\end{align*}
We call $t$ an \emph{absorption term}, and denote absorption by $B\absorbs \algA$. 
\end{definition}

We also say that $B$ is an absorbing subuniverse of $\algA$, that $t$ witnesses
the absorption $B \absorbs \algA$, etc. We call an absorbing subuniverse of a
power of $\algA$ an \emph{absorbing subpower} of $\algA$.


Two absorptions can be witnessed by a common absorption term~\cite{barto-kozik-cyclic-terms-and-csp}: If $B \absorbs\algA$ by a term $t$ and $C\absorbs \algA$ by
a term $s$, then
\[
  s\star t= s(t(x_{11},\dots,x_{1n}),\dots,t(x_{m1},\dots,x_{mn})).
\]
 is a common absorption term for both $B\absorbs\algA$ and $C\absorbs \algA$.

As the following well known proposition 
shows, absorption can be understood as a generalization of near unanimity.

\begin{proposition} \label{abs-and-nu}
A finite idempotent algebra $\algA$ has an NU term if and
only if every singleton absorbs $\algA$.
\end{proposition}
\begin{proof}
  In one direction, an NU term is an absorption term for any $\{a\}\leq
  \algA$.
  To obtain the other implication, we repeat the $s\star t$ construction until
  we arrive at a term $n$ which is an absorption term for every $\{a\}\absorbs
  \algA$. Examining the absorption equations, we see that 
  $n$ satisfies the NU identities for any choice of $x,y$.
\end{proof}

\subsection{J\'onsson absorption}
A weaker, but still useful~\cite{cd-implies-bw}, kind of absorption is
 Jónsson  absorption. 
 
\begin{definition} \label{def-jonsson-abs}
Let $B\leq\algA$. We say that $B$ \emph{J\'onsson absorbs}
 $\algA$, written $B \Jabsorbs \algA$, if there exist  terms (a \emph{J\'onsson absorption chain}) $d_0,d_1,\dots,d_n$ of $\algA$ such that:
\begin{align*}
  \forall i=0,\dots,n,\,d_i(B,A,B)&\subset B\\
  \forall i=0,\dots,n-1,\,d_i(x,y,y)&=d_{i+1}(x,x,y)\\
  d_0(x,y,z)&=x\\
  d_n(x,y,z)&=z.
\end{align*}
\end{definition}
Note that these terms do not correspond to the standard Jónsson terms, but to
directed J\'onsson terms introduced in~\cite{directed-jonsson}. However, the definition via the original J\'onsson terms gives an equivalent concept of J\'onsson absorption~\cite{directed-jonsson}.

J\'onsson absorption is
weaker than absorption:

\begin{proposition} \label{abs-jabs}
If $B \absorbs \algA$, then $B \Jabsorbs \algA$.
\end{proposition}
\begin{proof}
If $B\absorbs \algA$ by a term
$t(x_1,\dots,x_n)$, then the terms
\begin{align*}
  d_0(x,y,z)&=x\\
  d_1(x,y,z)&=t(y,x,x,\dots,x,x)\\
  d_2(x,y,z)&=t(z,y,x,\dots,x,x)\\
  d_3(x,y,z)&=t(z,z,y,\dots,x,x)\\
	    &\vdots\\
  d_{n-1}(x,y,z)&=t(z,z,z,\dots,z,y)\\
  d_n(x,y,z)&=z
\end{align*}
witness that  $B\Jabsorbs \algA$.
\end{proof}

  The other implication does not hold: Let us define $\algA$ as the full idempotent reduct of
  the implication algebra on $\{0,1\}$. Here, $\{0\}\Jabsorbs \algA$ is
  witnessed by the Jónsson absorbing terms
  \begin{align*}
    d_0(x,y,z)&=x,\\
    d_1(x,y,z)&=(y\to(z\to x))\to x,\\
    d_2(x,y,z)&=(x\to(y\to z))\to z,\\
    d_3(x,y,z)&=z.
  \end{align*}
  However, as we will see in Subsection~\ref{secBlockers},
  $\{0\}$ is not an absorbing subuniverse of $\algA$ because this algebra has a
  $\{0\}$-blocker.

  There are numerous parallels between absorption and J\'onsson absorption.
Similarly to absorption, we have the following result.

\begin{lemma} \label{common-jabs}
  Assume that $B\Jabsorbs \algA$ via the sequence of terms $\{d_i\}_{i=0}^k$ and 
  $C\Jabsorbs \algA$ by the sequence of terms $\{e_i\}_{i=0}^l$. Then there
  exists a sequence of terms $\{f_i\}_{i=0}^{kl+k}$ that is J\'onsson absorbing for
  both $B\Jabsorbs \algA$ and $C\Jabsorbs \algA$.
\end{lemma}
\begin{proof}
  It is straightforward to verify that the following chain of terms
  $f_0$, $f_1$, \dots, $f_{k(\ell+1)}$ will witness both $B\Jabsorbs \algA$ and $C\Jabsorbs \algA$:
  For $0\leq a\leq k$, $0\leq b \leq \ell$, we let
\[
  f_{a(\ell+1)+b}(x,y,z)=d_a(x,e_b(x,y,z),z).
\]
\end{proof}

Since an algebra generates a congruence distributive variety if and only if it has J\'onsson terms~\cite{jonsson} if and only if it has
directed J\'onsson terms~\cite{directed-jonsson}, we have the following consequence.
\begin{corollary} \label{cor-Jonsson}
  A finite idempotent algebra $\algA$ lies in a congruence distributive variety (equivalently, it has (directed) J\'onsson terms)  if and only if 
$\{a\}\Jabsorbs \algA$ for every $a\in A$.
\end{corollary}
Now we demonstrate how absorption allows us to transfer
connectivity properties to subalgebras. 

A subuniverse $E$ of $\algA^2$ is often regarded as a digraph $(A,E)$ with vertex set $A$ and edge set $E$. 
By a \emph{fence} of length $k$ from $a$ to $a'$ in $E$
we mean a sequence of forward and backward edges in this digraph, that is,
a sequence \[ a=e_1,e_2,e_3,\dots,e_{k},e_{k+1}=a', \]  where $(e_i,e_{i+1})\in
E$ for $i$ odd, and $(e_{i+1},e_i)\in E$ for $i$ even. Note that the length of
a fence is the number of edges, not vertices.

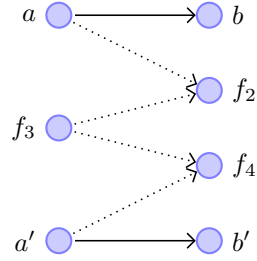
\begin{figure}
  \begin{center}
    \begin{tikzpicture}
    \node[vertex,label=left:$a$]  (a) {};
    \node[vertex, right of=a, node distance=2cm,label=right:$b$]  (b) {};
    \node[vertex, below of=a, node distance=3cm, label=left:$a'$]  (ap) {};
    \node[vertex, right of=ap, node distance=2cm,label=right:$b'$]  (bp) {};
    \node[vertex, below of=b, label=right:$f_2$]  (c1) {};
    \node[vertex, below of=c1,label=right:$f_4$]  (c2) {};
    \node[vertex, below of=a, node distance=1.5cm,label=left:$f_3$]  (c3) {};

    \draw[diedge]
    (a) edge (b)
    (ap) edge (bp)
    (a) edge[dotted] (c1)
    (c3) edge[dotted] (c1)
    (c3) edge[dotted] (c2)
    (ap) edge[dotted] (c2);
  \end{tikzpicture}
  \end{center}
  \caption{The situation in Lemma~\ref{lemFence} for $k=4$. The solid lines correspond to tuples
  from $E$, while the dotted lines denote tuples in $F$.}
  \label{figCDpath}
\end{figure}

\begin{lemma}\label{lemFence}
  Assume that $\alg{E}, \alg{F} \leq \algA^2$, that
  we have $(a,b),(a',b')\in E$ for some $a,a',b,b'\in A$, 
  that $\alg{E} \Jabsorbs \alg{F}$, and that there is a
  fence of even length from $a$ to $a'$ in $F$. Then there exists 
  a fence of even length from $a$ to $a'$ in $E$ as well.
\end{lemma}
\begin{proof}
  Let $d_0,d_1,\dots,d_n$ be a Jónsson absorbing chain witnessing $\alg{E} \Jabsorbs \alg{F}$.
  For each $i$ in $1,\dots, n-1$, we construct an even legth fence from $d_i(a,a,a')$ to
  $d_i(a,a',a')$. Since $d_i(a,a',a')=d_{i+1}(a,a,a')$, we can then concatenate
  those fences to obtain a long fence from $d_1(a,a,a')=a$ to
  $d_{n-1}(a,a',a')=a'$.

  Let $a,f_2,f_3,\dots,f_{k-1},a'$ be a fence from $a$ to $a'$ in $F$. Consider the matrix
  \[
  \begin{pmatrix}
    a&b&a&b&\dots&b&a\\
    a&f_2&f_3&f_4&\dots&f_{k-1}&a'\\	
    a'&b'&a'&b'&\dots&b'&a'\\
  \end{pmatrix}.
\]
The top and bottom rows are fences in $E$, the middle row is a fence in $F$.
Applying $d_i$ to the columns of this matrix and using $\alg{E}\Jabsorbs \alg{F}$, we get a
fence in $E$ from $d_i(a,a,a')$ to $d_i(a,a',a')$, as required.
\end{proof}

By replacing fences by directed paths and using the same construction, we get the
following lemma.

\begin{lemma}\label{lemDirectedConnectivity}
  Assume that $\alg{E},\alg{F}\leq \algA^2$, that $(a,a),(b,b)\in E$ for some $a,b\in
  A$, that $\alg{E}\Jabsorbs \alg{F}$, and that there is a
  directed path from $a$ to $b$ in the digraph $(A,F)$. Then there exists 
  a directed path from $a$ to $b$ in $(A,E)$ as well.
\end{lemma}

\subsubsection{Absorption and pp-definitions}

It is well known that the set of subpowers of an algebra is closed under primitive positive (pp-) definitions. 
The following lemma is a version of this fact for absorption and J\'onsson absorption.

\begin{lemma}\label{lemAbsPP}
  Assume that a subpower $R$ of $\algA$ is defined by 
  \[
    R=\{(x_1,\dots,x_n) \colon \exists y_1,\dots,y_m,\, R_1(\sigma_1)\AND
    R_2(\sigma_2)\AND\dots\AND R_k(\sigma_k)\},
  \]
  where $R_1,\dots,R_k$ are subpowers of $\algA$ regarded as predicates
   and $\sigma_1$,\dots,$\sigma_k$ stand for sequences of (free or bound) variables. 
  Let $S_1,\dots,S_k$ be subpowers of $\algA$ such that $S_i\absorbs \alg{R}_i$ (resp.
  $S_i\Jabsorbs \alg{R}_i$) for all $i$. Then the subpower
  \[
    S=\{(x_1,\dots,x_n)\colon \exists y_1,\dots,y_m,\, S_1(\sigma_1)\AND
    S_2(\sigma_2)\AND\dots\AND S_k(\sigma_k)\},
  \]
  absorbs (resp. Jónsson absorbs) $\alg{R}$.
\end{lemma}
\begin{proof}
  Let $t$ be a $k$-ary common absorption term for all $S_i\absorbs \alg{R}_i$ (see the remark above Proposition~\ref{abs-and-nu} and Lemma~\ref{common-jabs}). We
  show that $t$ is also an absorption term for $S\absorbs \alg{R}$. For
  simplicity, we verify this fact only in the case when the exceptional
  argument of $t$ is the last one. So, assume that
  $r_{i,j}$ are members of $A$ such that
  $(r_{j,1},\dots,r_{j,n})\in S$ for each $j=1,2,\dots,k-1$ and
  $(r_{k,1},\dots,r_{k,n})\in R$. Examining the formulas for $S$ and $R$, we see that
  for each $j=1,2,\dots,k-1$ there exist $y_{j,1},\dots,y_{j,m}$ so that 
  $(r_{j,1},\dots,r_{j,n},y_{j,1},\dots,y_{j,m})$ satisfy the conjunction
  $S_1(\sigma_1)\AND S_2(\sigma_2)\AND\dots\AND S_k(\sigma_k)$, while for $j=k$
  we have $y_{k,1},\dots,y_{k,m}$ such that
  $(r_{k,1},\dots,r_{k,n},y_{k,1},\dots,y_{k,m})$ satisfy
  $R_1(\sigma_1)\AND R_2(\sigma_2)\AND\dots\AND R_k(\sigma_k)$.

  We consider the matrix
  \[
    \begin{pmatrix}
      r_{1,1}&r_{1,2}&\dots&r_{1,n-1}&r_{1,n}&y_{1,1}&y_{1,2}&\dots&y_{1,m}\\
      r_{2,1}&r_{2,2}&\dots&r_{2,n-1}&r_{2,n}&y_{2,1}&y_{2,2}&\dots&y_{2,m}\\
      &&\vdots\\                                
      r_{k,1}&r_{k,2}&\dots&r_{k,n-1}&r_{k,n}&y_{k,1}&y_{k,2}&\dots&y_{k,m}\\
    \end{pmatrix}.
  \]

  Applying $t$ to the columns of this matrix yields a tuple
  $(s_1,\dots,s_n,y'_1,\dots,y'_m)$. Using $S_i\absorbs \alg{R}_i$, it is
  straightforward to show that this tuple satisfies each conjunct
  $S_i(\sigma_i)$ of the
  formula
  $S_1(\sigma_1)\AND \dots\AND S_k(\sigma_k)$.

  The proof for J\'onsson absorption is similar.
\end{proof}

Since we can use the above lemma in the case when
$\alg{R}_i=\algA^n$, we immediately obtain that the set of (J\'onsson) absorbing subpowers
of a given algebra is closed under primitive positive definitions, as long as
those definitions do not use the equality relation. There is a good reason for excluding the equality relation:
%
\begin{proposition}
  Let $\algA$ be a finite algebra, $|A|>1$. Then the relation $\Delta=\{(a,a) \colon a\in A\}$
  is never a J\'onsson absorbing subpower of $\algA$.
\end{proposition}
\begin{proof}
  Assume that there is a chain $d_0,\dots,d_n$ of J\'onsson absorbing terms for
  $\Delta\Jabsorbs \algA^2$. 
  Then for every $a,b,c,d$ in $A$ and each $d_i$ we have
  \begin{align*}
  d_i\left(\begin{pmatrix}a\\a\end{pmatrix},\begin{pmatrix}b\\c\end{pmatrix},\begin{pmatrix}d\\d\end{pmatrix}\right)\in\Delta.
  \end{align*}
  Written in a different way, $d_i(a,b,d)=d_i(a,c,d)$ for all $a,b,c,d$. That
  is,
  $d_i$ does not depend on its second input. But then from the equations
  $d_i(x,y,y)=d_{i+1}(x,x,y)$ we get
  $x=d_0(x,y,z)=d_1(x,y,z)=\dots=d_n(x,y,z)=z$, so the algebra
  $\algA$ must be trivial.
\end{proof}

\subsection{Witnessing lack of absorption}\label{secBlockers}
%
By Proposition~\ref{abs-jabs},
one way to witness that $B$ does not absorb $\algA$ is to show that $B$ does not J\'onsson absorb $\algA$.
Another way is to find a $B$-blocker, which is a special
kind of a cube term blocker, introduced
in~\cite{markovic-maroti-mckenzie-blockers}:
\begin{definition} \label{def-blocker}
  Let $B\leq \algA$. We say that $(C,D)$  is a \emph{$B$-blocker} of $\algA$ if
\begin{enumerate}
  \item $D \leq \algA$,
  \item $\emptyset\neq C\subset D$,
  \item $C\cap B=\emptyset$,
  \item $D\cap B\neq\emptyset$, and
  \item \label{itmCondBBlocker} $\forall n,  D^n\setminus (D\setminus C)^n \leq \algA^n$ (see
  Figure~\ref{figBBlocker}).
\end{enumerate}
\end{definition}

\begin{figure}
  \begin{center}
    \begin{tikzpicture}[node distance=2cm]
      \node [vertex,label=below:${(c,c)}$] (cc) {};
      \node [vertex,label=below:${(c,d)}$,above of=cc](cd) {};
      \node [vertex,label=below:${(d,c)}$,right of=cc](dc) {};
      \node [vertex,label=below:${(d,d)}$,above of=dc](dd) {};
      \node [inner sep=0,minimum size=0] at ($(cc)+(-1.4,-1.4)$) (CC) {};
      \node [inner sep=0,minimum size=0]at ($(cd)+(-1.4,1.4)$) (CD) {};
      \node [inner sep=0,minimum size=0,label=below
      right:$D^2\setminus (D\setminus C)^2$] at ($(dc)+(1.4,-1.4)$) (DC) {};
      \node [inner sep=0,minimum size=0] at ($(cd)!0.5!(dc)$) (DD1){};
      \node [inner sep=0,minimum size=0] at ($(DD1)+(2.4cm,0)$) (DD2){};
      \node [inner sep=0,minimum size=0] at  ($(DD1)+(0,2.4cm)$)(DD0){};
      \node at ($(cd)!0.5!(cc)+(-2,0)$) (D1){$D$};
      \node at ($(dc)!0.5!(cc)+(0,-2)$) (D2){$D$};
      \node at ($(DD2)!0.5!(DC)+(0.3,0)$) (C1){$C$};
      \node at ($(CD)!0.5!(DD0)+(0,0.3)$) (C2){$C$};
      \node [above right of=dd, node distance=1.8cm] {$(B\cap D)^2$};
      \draw  (CD)--(DD0)--(DD1)--(DD2)--(DC)--(CC)--(CD);
      \draw (dd) circle (0.9cm);
  \end{tikzpicture}
  \end{center}
  \caption{The relations $D^n\setminus (D\setminus C)^n$ (square with a corner
    cut off) and $(B\cap
    D)^n$ (disc) for $n=2$.}
  \label{figBBlocker}
\end{figure}
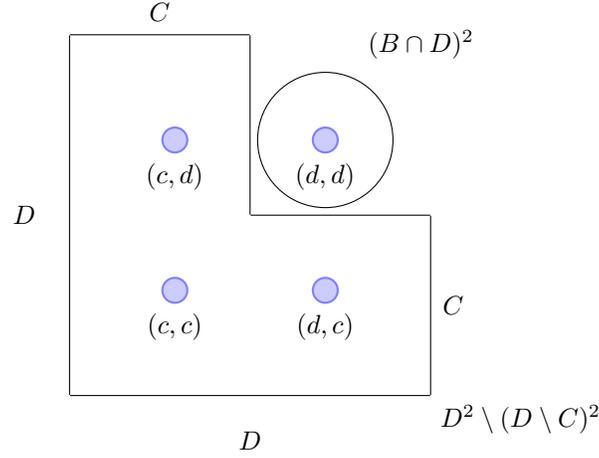

Observe that the last property for $n=1$ says that $C$ is a subuniverse of $\algA$.

We remark that $(C,D)$ is a cube term blocker in the sense of~\cite{markovic-maroti-mckenzie-blockers}
if and only if it is a $\{b\}$-blocker for some $b \in A$. 

\begin{proposition} \label{blocker-noabs}
Let $B \leq \algA$. If $\algA$ has a $B$-blocker, then $B$ does not absorb $\algA$.
\end{proposition}
\begin{proof}
  Assume that there exists a $B$-blocker $(C,D)$, but $B \absorbs \algA$ by an $n$-ary term $t$. Let $c$ be some
  member of $C$ and $d$ a member of $D\cap B$. Consider the $n \times n$ matrix
\[
  \left(\begin{matrix}
      c&d&d&\dots&d&d\\
      d&c&d&\dots&d&d\\
      d&d&c&\dots&d&d\\
       &\vdots\\
      d&d&d&\dots&c&d\\
      d&d&d&\dots&d&c\\
  \end{matrix}\right).
\]
Each row lies in $D^n \setminus (D \setminus C)^n \leq \algA^n$, therefore $t$ applied to the columns is an $n$-tuple $\tupl{a}$ from $D^n \setminus (D \setminus C)^n$. 
But each column consists of elements in $D \leq \algA$ and all but one of the entries in each column are from $B \absorbs \algA$. It follows that $\tupl{a} \in (B \cap D)^n \subseteq (D \setminus C)^n$, a contradiction.   
\end{proof}

The absence of $B$-blockers by itself does not guarantee that $B$ is absorbing:
Consider the full idempotent reduct $\algA$ of the group $\zet_2$ and choose
$B=\{0\}$. Since $\zet_2$ has a Malcev term $m(x,y,z)=x+y+z\pmod 2$, there is
no $B$-blocker of $\algA$. However, were some $t(x_1,\dots,x_n)=\sum a_i x_i$ an absorption term for
$\{0\}\absorbs \algA$, we would run into trouble when trying to satisfy the
following set of equalities simultaneously:
\begin{align*}
  t(1,0,0,\dots,0)&=0,\\
  t(0,1,0,\dots,0)&=0,\\
  t(0,0,1,\dots,0)&=0,\\
  \vdots\\
  t(0,0,0,\dots,1)&=0.
\end{align*}
The only linear term that satisfies these equalities is the constant zero term, which is not idempotent.

In the arguments to come, it will be useful to have a description of
absorption that talks about subpowers rather then terms. 
If $J\subset I$ are sets of indices and $R\leq \algA^I$, we denote by $\pi_{\widehat J}(R)$ 
the projection of $R$ onto the coordinates in $I\setminus J$. For $J = \{j\}$,
we write simply $\pi_{\widehat j}(R)$.

\begin{definition}
A subuniverse $R$ of $\algA^k$ is called \emph{$B$-essential} if $R\cap B^k=\emptyset$ and, for every $i=1,2,\dots k$, 
we have $\pi_{\widehat i} (R)\cap B^{k-1}\neq \emptyset$.
\end{definition}

Observe that if $R \leq \algA^k$ is $B$-essential, then by fixing the first coordinate to $B$, that is, by defining
\[
\{(a_2, \dots, a_k) \colon \exists b \in B,\, (b,a_2, \dots, a_k) \in R\},
\] 
we get a $B$-essential subuniverse of $\algA^{k-1}$.  
Therefore, the set of arities of $B$-essential subpowers of an algebra is a downset in $\mathbb{N}$.

\begin{proposition}\label{propEssentialRel}
Let $\algA$ be a finite algebra and $B \leq \algA$. Then $\algA$ has a $k$-ary
absorption term witnessing $B\absorbs\algA$ if and only if $\algA$ has no $k$-ary essential
subpower. In particular, $B\absorbs\algA$ if and only if there exists a $k$
such that $\algA$ has no $k$-ary essential subower.
\end{proposition}

\begin{proof}
  Assume that $B\absorbs\algA$ by a $k$-ary absorption term $t$ and 
  assume that $R\leq \algA^k$ satisfies $\forall i, \pi_{\widehat i} (R)\cap
  B^{k-1}\neq \emptyset$. Then there exist $b_{ij}\in B$ and $a_i\in A$ so that
  the rows of the matrix
  \[
  \begin{pmatrix}
  a_1&b_{1,2}&b_{1,3}&\dots&b_{1,k}\\
  b_{2,1}&a_2&b_{2,3}&\dots&b_{2,k}\\
    b_{2,1}&b_{2,2}&a_3&\dots&b_{2,k}\\
		       &&&\ddots&\\
  b_{k,1}&b_{k,2}&b_{k,3}&\dots&a_{k}
  \end{pmatrix}
  \]
are all in $R$. By applying $t$ component-wise to the
colums of this matrix, we get a $k$-tuple that lies in $R\cap B^k$. Consequently, $\algA$ has no $B$-essential subpower of arity $k$.

The other implication is more interesting. Let $B\leq\algA$ and assume that
$\algA$ has no $k$-ary essential subpower. We shall show how to produce a $k$-ary absorption
term. 


We begin by extending the above statement to a bigger family of subpowers:
Assume that $n$ is a positive integer, $R\leq \algA^n$ and that there exists a family of
pairwise disjoint sets $I_1,\dots,I_k\subset \{1,2,\dots,n\}$ such that
$\pi_{\widehat{I_j}}(R)\cap B^{n-|I_j|}\neq \emptyset$. We claim that then $R\cap B^n\neq
\emptyset$. This can be easily proved by induction on $n$: For $n<k$, the claim
is trivial, while the case $n=k$ is equivalent to saying that $\algA$ has no $k$-ary
essential subpower. Assume now that the claim holds
for $n-1$ and show how to extend it to $n$: By the induction hypothesis, 
the relation $\pi_{\widehat{i}}(R)$ intersects with $B^{n-1}$ for each
$i=1,\dots,n$, so if $B^n\cap R=\emptyset$, then $R$ is an $n$-ary $B$-essential relation. 
However, the set of arities of $B$-essential relations is a downset and we know
that there is no $k$-ary $B$-essential relation, so $R$ cannot be $B$-essential
and we get $B^n\cap R\neq\emptyset$.


Consider now the $k|B|^{k-1}(|A|-|B|)$-ary relation 
\[
  R=\{(t(e_1,e_2,\dots,e_k))_{(e_1,\dots,e_k)\in I} \colon
\text{$t$ is a $k$-ary term of $\algA$}\},
\]
where $I$ is the set of all $k$-tuples $(e_1,e_2,\dots,e_k)\in A^k$ that
contain exactly one member of $A\setminus B$ and the remaining elements are from $B$.

From the definition of $R$ it follows that $R$ is a subuniverse of $\algA^I$. Moreover, if 
we let $I_j=\{(e_1,\dots,e_n) \colon e_j\in A\setminus B,\, \forall i\neq j,\, e_i\in B\}$,
we obtain a partition of $I$  such that 
$\pi_{\widehat{I_j}}(R)\cap B^{I\setminus I_j} \neq \emptyset$ (consider
the operation of projection to the $j$-th coordinate). Therefore, by the previous claim, $R$
 contains a tuple consisting entirely of elements of $B$. This tuple corresponds to a $k$-ary
absorption term for $B\absorbs \algA$.
\end{proof}

It will be convenient to extend the notion of $B$-essential subpower to countably infinitary powers:  A subuniverse $R$ of $\algA^{\en}$ is called $B$-essential if $R\cap B^{\en}=\emptyset$ and, for every $i \in \en$, 
we have $\pi_{\widehat i} (R)\cap B^{\en}\neq \emptyset$. 



\section{Absorption $=$ J\'onsson absorption $+$ no $B$-blockers}

We are ready to prove the main result, Theorem~\ref{MainThm}.
The implication (i) $\Rightarrow$ (ii) follows from Proposition~\ref{abs-jabs} and Proposition~\ref{blocker-noabs}.
Therefore, it is enough to prove the following theorem. 

\begin{theorem}\label{thmAbsChar}
Let $\algA$ be a finite idempotent algebra, $B\leq \algA$, $B\Jabsorbs \algA$, and assume that $\algA$ has no $B$-blocker. Then $B\absorbs \algA$.
\end{theorem}
The rest of this section is devoted to the proof of the above statement. We take a counterexample to the theorem with $|A|$ minimal.
Thus $B$ is a subuniverse of $\algA$, $B\Jabsorbs \algA$, $\algA$ has no $B$-blocker, $B$ does not
absorb $\algA$, and $\algA$ is of minimal size among algebras with such a subuniverse $B$. 

By Proposition~\ref{propEssentialRel}, there are $B$-essential subpowers of
$\algA$ of arbitrary large arity. We will show that $\algA$ has a family of
highly symmetric $B$-essential subpowers of countably infinite arity, which we
use to produce a $B$-blocker of $\algA$.

We begin by showing that the $B$-essential relations can be chosen to be of a rather special kind. 
 In the
following sequence of lemmas, we will be using two wild card symbols: The letter
$\alpha$ stands for any suitable member of $A$ and $\beta$ stands for any
suitable member of $B$.
As an example, the meaning of the statement ``$(\beta,\beta)\in R$'' is ``there
exist $b_1,b_2\in B$ (which can be equal or different) such that $(b_1,b_2)\in R$''.

\begin{lemma}\label{lemRamsey}
  There exists a function $f:\en\to \en$ such that if $\algA$ has  an $f(N)$-ary $B$-essential
  subpower, then there exist $b_1,b_2\in B$ and $a_{1},\dots,a_{N} \in A$ such
  that the subuniverse of $\algA^N$ generated by the $N$-tuples
  \begin{align*}
    (a_1,b_2,b_2,b_2,&\dots,b_2,b_2)\\
    (b_1,a_2,b_2,b_2,&\dots,b_2,b_2)\\
    (b_1,b_1,a_3,b_2,&\dots,b_2,b_2)\\
    \vdots&\\
    (b_1,b_1,b_1,b_1,&\dots,a_{N-1},b_2)\\
  (b_1,b_1,b_1,b_1,&\dots,b_1,a_N).
  \end{align*}
is also $B$-essential.
\end{lemma}

\begin{proof}
  This is a Ramsey-type theorem and we will use Ramsey's theorem (see eg.~\cite[Theorem 3.3]{a-course-in-combinatorics}) to prove it.
  Given $N$, we choose the number $f(N)$ so that whenever we have a complete symmetric
  graph $\relstr{G}$ with at least $f(N)$ vertices whose
  edges have been colored by $|B|^2$ colors, then there is a monochromatic
  complete subgraph of $\relstr{G}$ with at least $N$ vertices. 
  
  Suppose that $S$ is a $B$-essential subpower of $\algA$ of arity $f(N)$.
  For every $i$ from $1$ to $N$, we pick a tuple in $S$ that witnesses
  $\pi_{\widehat{i}}(S)\cap B^{N-1} \neq \emptyset$, and write these tuples as rows of an
  $f(N)\times f(N)$ matrix $M$. Using the wild cards $\alpha$ and $\beta$, we
  know that $M$ has the form
  \[M=
    \begin{pmatrix}
      \alpha & \beta & \beta &\dots &\beta\\
      \beta & \alpha & \beta &\dots &\beta\\
      \beta & \beta & \alpha &\dots &\beta\\
		   &&\vdots\\
      \beta &\beta&\beta&\dots&\alpha
  \end{pmatrix}.
  \]
  
  We will denote by $M_{ij}$ the entry of $M$ in the $i$-th row and
  $j$-th column.
  We now take the complete graph $\relstr{G}$ with the vertex set
  $\{1,2,\dots,f(N)\}$ and color the edge between $i$ and $j$, where $i<j$, by the pair $(M_{ij},M_{ji})$.
  By Ramsey's theorem, there exists a set of indices $I$ of size 
  $N$ such that the subgraph of $\relstr{G}$ induced by $I$ has monochromatic edges.
  What does this mean for $M$? If we look only at the rows and columns of $M$ with indices in
  $I$, we get an $N\times N$ submatrix of $M$ that looks like
   \[
    M'=\begin{pmatrix}
      a_1 & b_2 & b_2 &\dots &b_2\\
      b_1 & a_2 & b_2 &\dots &b_2\\
      b_1 & b_1 & a_3 &\dots &b_2\\
		   &&\vdots\\
      b_1 & b_1& b_1&\dots& a_N
  \end{pmatrix}
  \]
  for some $b_1,b_2\in B$.
  The rows of $M'$ have exactly the right form for the conclusion of the
  lemma. To finish the proof, we claim that the rows of $M'$ generate a $B$-essential subpower $T$ of $\algA$. 
  
  The generators themselves ensure that
  $\pi_{\widehat{i}}(T)\cap B^{N-1}\neq \emptyset$. Since we have obtained $M'$
  by restricting $M$ to indices $I\times I$, the relation $T$ is a subset of the subpower
  \[
    \pi_{I}(S\cap \{(s_1,\dots,s_{f(N)}) \colon s_i\in B\; \text{for all}\, i\not \in
    I\}),
  \]
  which means that $T\cap B^N=\emptyset$, and $T$ is $B$-essential.
\end{proof}

\begin{lemma}\label{lemSymmetric}
  For every 
  $K\in\en$ there exist $a\in A$ and $b\in B$ such that the
  subuniverse of $\algA^K$  generated by the tuples
  \[
  \begin{matrix}
    (a,&b,&\dots,&b,&b)\\
    (b,&a,&\dots,&b,&b)\\
       &&\ddots\\
    (b,&b,&\dots,&a,&b)\\
    (b,&b,&\dots,&b, &a)\\
\end{matrix}
\]
is $B$-essential.
\end{lemma}
\begin{proof}
   First observe that it is enough to find $b\in B$ such that for
  some $a_1,\dots, a_{|A|K}$, the $|A|K$-ary subpower
  $S$ of $\algA$ generated by tuples of the form
  \[
  \begin{matrix}
    (a_1,&b,&\dots,&b,&b)\\
    (b,&a_2,&\dots,&b,&b)\\
       &&\ddots\\
    (b,&b,&\dots,&a_{|A|K-1},&b)\\
    (b,&b,&\dots,&b, &a_{|A|K})\\
\end{matrix}
\]
is $B$-essential. Indeed, by the Dirichlet principle, 
there has to be some $a\in A$ that appears at least
  $K$ times on the main diagonal of the above matrix. Assume without loss of
  generality that this $a$ appears in the first $K$ columns of the
  diagonal. We now take the $K$-ary relation $T \leq \algA^K$ generated by the first $K$ elements of
  the first $K$ rows of the matrix of generators of $S$, that is, $T$ is generated by
  \[
  \begin{matrix}
    (a,&b,&\dots,&b,&b)\\
    (b,&a,&\dots,&b,&b)\\
       &&\ddots\\
    (b,&b,&\dots,&a,&b)\\
    (b,&b,&\dots,&b, &a)
\end{matrix}
\]
  As in the proof of Lemma~\ref{lemRamsey}, $T$ turns out to be $B$-essential.

 The idea for the remainder of the proof is to force $b_1$, $b_2$ from Lemma~\ref{lemRamsey} to be equal.
  To this end, we call a pair $(b_1,b_2)\in B^2$ \emph{good} if for
 $N=|\langle b_1,b_2\rangle|(|A|K+1)$ there exist $a_1,\dots,a_N\in A$ such
 that the $N$-ary subpower $R$ of $\algA$ generated
  by the tuples
  \[
  \begin{matrix}
    (a_1,&b_2,&b_2,&b_2,&\dots,&b_2,&b_2),&\\
    (b_1,&a_2,&b_2,&b_2,&\dots,&b_2,&b_2),&\\
    (b_1,&b_1,&a_3,&b_2,&\dots,&b_2,&b_2),&\\
		       &&&&\ddots&\\
    (b_1,&b_1,&b_1,&b_1,&\dots,&a_{N-1},&b_2),&\\
  (b_1,&b_1,&b_1,&b_1,&\dots,&b_1,&a_N)
  \end{matrix}
\]
  is $B$-essential. Since $\algA$ has $B$-essential relations of all arities,  Lemma~\ref{lemRamsey} gives us a pair $(b_1,b_2)\in B^2$
  that works for some $N\geq |\langle b_1,b_2\rangle|(|A|K+1)$. By fixing 
  coordinates, we get that $(b_1,b_2)$ is a good pair.
  We now take  a good pair $(b_1,b_2)$ for which $n=|\langle b_1,b_2\rangle|$ is
  minimal. If $n=1$, then $b_1=b_2$ and we are done.

In the rest of the proof, we show that if $n>1$, then there exists a good pair
$(b_1,b_2)$ such that $|\langle b_1,b_2\rangle|<n$. 

Consider the digraph $\relstr{G}=(B,E)$ with the vertex set $B$, where we let $(i,j)\in E$ if
and only if for every choice of $k$ from 0 to $|A|K-1$ and $\ell$ from 0 to $(n-1)(|A|K+1)-1$,
we have
\[  
  (\underbrace{b_1,\dots,b_1}_{k},\alpha,\underbrace{i,i,\dots,i}_{|A|K-1-k},\beta,\underbrace{j,j,\dots\dots\dots\dots,j}_{(n-1)(|A|K+1)-1-\ell},\alpha,\underbrace{b_2,\dots,b_2}_{\ell})\in
  R.
\]
Since $R$ contains the tuples $(b_1,\dots,b_1,a_r,b_2,\dots,b_2)$, the digraph
$\relstr{G}$ has loops around $b_1$ and $b_2$.  
We claim
that there exists a directed path from $b_1$ to $b_2$ in $\relstr{G}$. Why is that the
case? Define a binary relation $F$ so that $(i,j)\in F$ if
and only if 
\[  
  (\underbrace{b_1,\dots,b_1}_{k},\alpha,\underbrace{i,i,\dots,i}_{|A|K-1-k},\alpha,\underbrace{j,j,\dots\dots\dots\dots,j}_{(n-1)(|A|K+1)-1-\ell},\alpha,\underbrace{b_2,\dots,b_2}_{\ell})\in
  R,
\]
where as before $k$ ranges from 0 to $|A|K-1$ while $\ell$ ranges from 0 to $(n-1)(|A|K+1)-1$.
(The only difference between $E$ and $F$ is that we
have replaced $\beta$ by $\alpha$ on the position $|A|K+1$ of the long tuple.) Both $E$ and $F$ are clearly subuniverses of $\algA^2$.

Since $B\Jabsorbs \algA$, Lemma~\ref{lemAbsPP} tells us that $E\Jabsorbs \alg{F}$.

Since $(b_1,\dots,b_1,a_{|A|K+1},b_2,\dots,b_2)\in R,$ we have $(b_1,b_2)\in F$. 
By Lemma~\ref{lemDirectedConnectivity}, there is a directed path from $b_1$ to $b_2$ in
$\relstr{G}$. Denote the length of the shortest such path
by $m$. Since $b_1\neq b_2$, we have $m>0$. 

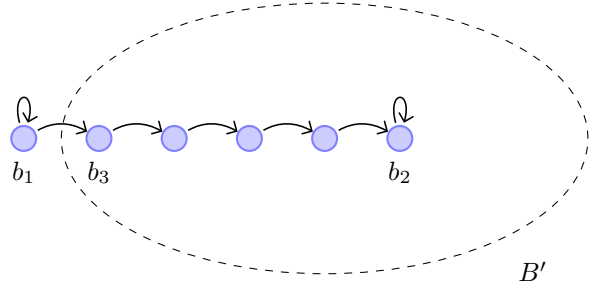
\begin{figure}
  \begin{center}
    \begin{tikzpicture}
      \node[vertex,label=below:$b_1$](b1){};
      \foreach \i/\j/\lab in {1/3/b_3,3/4/,4/5/,5/6/,6/2/b_2} {
	\node[vertex, label=below:$\lab$, right of=b\i] (b\j) {};
        \draw [diedge, bend left] 
	(b\i) edge (b\j);
      }
      \draw [diedge, loop above]
        (b1) edge (b1)
	(b2) edge (b2);
      \draw [dashed] (b6) ellipse (3.5cm and 1.8cm);
      \node [below right of=b2, node distance=2.5cm] {$B'$};
  \end{tikzpicture}
  \end{center}
  \caption{Obtaining $b_3$.}
  \label{figPath}
\end{figure}

Let $b_3$
be the second vertex on the shortest directed path from $b_1$ to $b_2$ in $\relstr{G}$
(see Figure~\ref{figPath}). Now consider the set $B'$
of vertices of $\relstr{G}$ that can
be reached from $b_2$ by going backwards $m-1$ times. Since $E$ is a subpower of $\alg{B}$, $B'$ is a subuniverse of $\alg{B}$. Moreover, $B'$ contains $b_2$ 
(because of the loop $(b_2,b_2)\in E$), as 
well as $b_3$, and $B'$ does not contain $b_1$ (else our path would not be
minimal). We therefore have $|\langle b_2,b_3\rangle|\leq |B'|<n$. We now show
that at least one of $(b_1,b_1)$ or $(b_3,b_2)$ is a good pair.

To that end, consider the relation 
\[
  S=\{(s_1,\dots,s_{|A|K}) \colon (s_1,\dots,s_{|A|K},\beta,\beta,\dots,\beta)\in
  R\}.
\]
It is straightforward to show that $S$ is a subpower of $\algA$ and that $S$ is $B$-essential of arity $|A|K$.
If we can find 
$(b_1,\dots,b_1,\alpha,b_1,\dots,b_1)\in S$
for $\alpha$ in every position, 
then $(b_1,b_1)$ is good and we are done. 

Assume therefore, that there exists a $k$ such that
\[
  (\underbrace{b_1,\dots,b_1}_{k},\alpha,\underbrace{b_1,\dots,b_1}_{|A|K-k-1})\not\in
S.
\]
Going back to the definition of $S$, we see that this is only possible if for
this value of $k$ we have
\[
  (\underbrace{b_1,\dots,b_1}_{k},\alpha,\underbrace{b_1,\dots,b_1}_{|A|K-k-1},\underbrace{\beta,
  \dots\dots\dots,\beta}_{(n-1)(|A|K+1)+1})\not\in
  R.
\]
Consider the relation
\[
  U=\{(u_1,\dots,u_{(n-1)(|A|K+1)}) \colon (\underbrace{b_1,\dots,b_1}_{k},\alpha,\underbrace{b_1,\dots,b_1}_{|A|K-k-1},\beta,u_1,\dots,u_{(n-1)(|A|K+1)})\in
  R\}.
\]
This is an $(n-1)(|A|K+1)$-ary subpower of $\algA$ that is disjoint from
$B^{(n-1)(|A|K+1)}$. Since $(b_1,b_3)$ is an edge in $\relstr{G}$, we have
\[  
 (\underbrace{b_1,\dots,b_1}_{k},\alpha,\underbrace{b_1,\dots,b_1}_{|A|K-k-1},\beta,\underbrace{b_3,\dots\dots\dots\dots,b_3}_{(n-1)(|A|K+1)-\ell-1},\alpha,\underbrace{b_2,\dots,b_2}_{\ell})\in
 R
\]
for every $\ell$ from 0 to $(n-1)(|A|K+1)-1$. Therefore, $U$ contains for
every applicable $\ell$ a tuple of the form
\[
  (\underbrace{b_3,\dots\dots\dots\dots,b_3}_{(n-1)(|A|K+1)-\ell-1},\alpha,\underbrace{b_2,\dots,b_2}_{\ell}).
\]
We see that $U$ is a $B$-essential relation of arity
$(n-1)(|A|K+1)$ and $|\langle b_2,b_3\rangle|\leq n-1$. By fixing some
coordinates of $U$ to $b_3$ to decrease the arity of $U$, we get a 
$(|\langle b_2,b_3\rangle|(|A|K+1))$-ary $B$-essential subpower $U'$ of $\algA$ that contains for each
$\ell$ a tuple of the form
\[
  (b_3,\dots\dots\dots,b_3,\alpha,\underbrace{b_2,\dots,b_2}_{\ell}).
\]
We can thus replace the pair $(b_1,b_2)$ by $(b_3,b_2)$ as we needed.
\end{proof}

Since there are only finitely many choices of $a,b$ and infinitely many choices
of $K$, we can apply Dirichlet's principle to make 
Lemma~\ref{lemSymmetric} stronger:

\begin{corollary}  \label{corSymmetric}
  There exist $a\in A,b\in B$ such that, for every $K\in\en$, the subuniverse of $\algA^K$ generated by
\[
  \begin{matrix}
    (a,&b,&\dots,&b,&b)\\
    (b,&a,&\dots,&b,&b)\\
       &&\ddots\\
    (b,&b,&\dots,&a,&b)\\
    (b,&b,&\dots,&b, &a)\\
\end{matrix}
\]
is $B$-essential.
\end{corollary}

For the remainder of the proof, we  fix a pair of
elements $a,b\in A$ from Corollary~\ref{corSymmetric}. It is an easy exercise to show that the (countably) infinitary subpower of $\algA$ generated by
\[
(a,b,b,b,\dots), (b,a,b,b,\dots), \dots
\] 
is $B$-essential.  We denote this
relation by $R_\infty$.

We say that a relation $R$ is \emph{symmetric}, if it is invariant under all 
permutations of coordinates. Given a $k$-ary relation $T\leq \algA^k$, 
we construct an infinitary subpower $S_T$ of $\algA$ by fixing the first $k$
coordinates of $R_\infty$ to members of $T$:
\[
  S_T=\{(s_1,s_2,\dots) \colon \exists (t_1,\dots,t_k)\in T,\,
  (t_1,t_2,\dots,t_k,s_1,s_2,\dots)\in R_\infty\}. 
\]
For $k=0$, we define $S_{A^0}=R_\infty$. 

From the choice of generators and the idempotency of $\algA$ we obtain the following observation. 

\begin{lemma}\label{lembbbbb}
  For every choice of $T$, $S_T$ is symmetric and only finitely many coordinates of any element of 
  $S_T$ are different from $b$.
\end{lemma}

By fixing  all but the first two coordinates of $S_T$ to elements of $B$ we
obtain the binary relation
\[
  S_{T;1,2}=\{(x,y) \colon (x,y,\beta,\beta,\beta,\dots)\in S_T\} \leq \algA^2.
\]
This (symmetric) binary relation defines a graph with the vertex set $A$. We denote the
neighborhood of the set $B$ in this graph  by $B^{+S_{T;1,2}}$:
\[
  B^{+S_{T;1,2}}=\{c\in A \colon (\beta,c)\in S_{T;1,2}\} = \{c \in A \colon (c,\beta) \in S_{T;1,2}\} \leq \algA.
\]


For some choices of $k$ and $T$ (such as $k=0$), the relation $S_T$ is
$B$-essential. We choose a subpower $T$ of $\algA$ so that $S_T$ is
$B$-essential and the size of the set $B^{+S_{T;1,2}}$ is maximal. We put
$C=B^{+S_{T;1,2}}$ and show that $(C,A)$ is a $B$-blocker. This will contradict the
choice of $B \leq \algA$ and finish the proof. 

The subuniverse $C$ is nonempty and disjoint from $B$ as $S_T$ is
$B$-essential. Except for the condition (\ref{itmCondBBlocker}), it is easy to 
verify that $(C,A)$ satisfies all conditions of
Definition~\ref{def-blocker}. We use a chain of lemmas to show that $A^n\setminus (A\setminus C)^n$ is
a subpower of $\algA$ for all choices of $n$.

\begin{lemma}\label{lemcinB}
    If $c\in C$, then $(b,c),(c,b)\in S_{T;1,2}$ and $\langle b,c \rangle = A$ (where $\langle b,c \rangle$ denotes the subuniverse of $A$ generated by $b$ and $c$).
\end{lemma}

\begin{proof}
From Lemma~\ref{lembbbbb} and the choice of $C$ it follows that both $(b,c)$ and $(c,b)$ are in $S_{T;1,2}$. 

To prove the second claim we show that $\langle b,c\rangle \subsetneq A$ leads
to a contradiction. Let $\algA'$ be the subalgebra of $\algA$ with universe $A' = \langle b,c \rangle$ and
$B' = B \cap \langle b,c \rangle$. Clearly, $B'$ is a J\'onsson absorbing
subuniverse of $\algA'$. We also know that $\algA'$ has no $B'$-blockers, since
every $B'$-blocker in $\algA'$ is also a $B$-blocker in $\algA$.

Since $\algA'$ is smaller than $\algA$ and we assume $\algA$ to be a minimial counterexample
to Theorem~\ref{thmAbsChar}, $B'$ absorbs $\algA'$. Call the absorbing term $t$ and let $k$ be its arity.
Since $c \in
B^{+S_{T;1,2}}$, $S_T$ contains a tuple of the form
$(c,\beta,\beta,\dots)$ which we can write as $(c,\beta,\dots,\beta,b,b,\dots)$
by Lemma~\ref{lembbbbb}. Therefore, by fixing some coordinates of $S_T$ to $B$,
we obtain a symmetric $B$-essential subuniverse $Q$ of $\algA^{\en}$ containing
all permutations of the tuple $(c,b,b,\dots)$. If we now apply $t$ to $k$ of these
tuples like in the proof of Proposition~\ref{propEssentialRel}, we get that
$Q\cap (B')^\en\neq \emptyset$, a contradiction with $Q$ being $B$-essential.
The only way to avoid this contradiction is if $\langle b,c\rangle =A$.
\end{proof}

Next we verify that $A^n\setminus (A\setminus C)^n\leq \algA$ for $n=2$.

  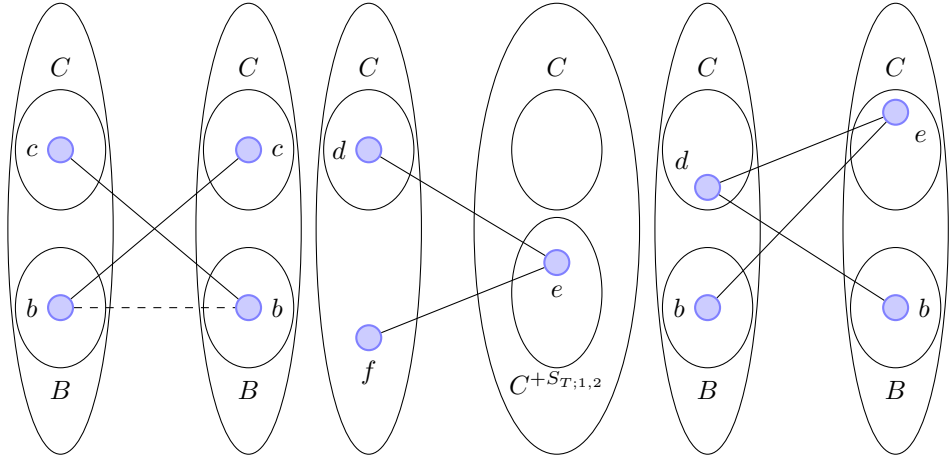
\begin{figure}
  \begin{center}
    \begin{tikzpicture}[node distance=2.1cm]
	\node[vertex, label=left:$c$] (cl) {};
	\node[vertex, label=right:$c$, right of=cl, node distance=2.5cm] (cr) {};
	\node[vertex, label=left:$b$, below of=cl] (bl) {};
	\node[vertex, label=right:$b$, right of=bl, node distance=2.5cm] (br) {};
	\node[radius=0pt, above of=cl] (el) {};
	\node[ right of=el, node distance=3cm] (er) {};
	\draw
	(cl) edge (br)
	(bl) edge (cr);
	\draw[dashed]
	(bl) edge (br);
	\draw 
	(bl) ellipse (6mm and 8mm)
	(br) ellipse (6mm and 8mm)
	 (cl) ellipse (6mm and 8mm) 
	(cr) ellipse (6mm and 8mm) 
	($(cl)!0.5!(bl)$) ellipse (7mm and 30mm)
	($(cr)!0.5!(br)$) ellipse (7mm and 30mm);
	\node [below of=br, node distance=1.1cm] {$B$};
	\node [below of=bl, node distance=1.1cm] {$B$};
	\node [above of=cr, node distance=1.1cm] {$C$};
	\node [above of=cl, node distance=1.1cm] {$C$};
  \end{tikzpicture}
  \hfil
  \begin{tikzpicture}
    \node (cl) {};
    \node[vertex, label=left:$d$] at (cl) (dl) {};
    \node[right of=cl, node distance=2.5cm] (cr) {};
    \node[vertex,label=below :$e$, below of=cr,node distance=1.5cm] (er) {};
    \node[vertex,label=below :$f$, below of=cl, node distance=2.5cm] (fl) {};
    \node[below of=cl, node distance=2.1cm] (bl) {};
    \node[right of=bl, node distance=2.5cm] (br) {};
    \draw
    (dl) edge (er)
    (fl) edge (er);
    \draw 
    ($(er)+(0,-4mm)$) ellipse (6mm and 10mm)
	(cl) ellipse (6mm and 8mm) 
        (cr) ellipse (6mm and 8mm) 
	($(cl)!0.5!(bl)$) ellipse (7mm and 30mm)
	($(cr)!0.5!(br)$) ellipse (11mm and 30mm);
	\node [below of=er, node distance=1.6cm] {$C^{+S_{T;1,2}}$};
	\node [above of=cr, node distance=1.1cm] {$C$};
	\node [above of=cl, node distance=1.1cm] {$C$};
\end{tikzpicture}
    \hfil
    \begin{tikzpicture}[node distance=2.1cm]
	\node (cl) {};
	\node[vertex, label=above left:$d$, below of=cl, node distance=5mm] (dl) {};
        \node[vertex, label=below right:$e$] at ($(cl)+(2.5cm,5mm)$) (er) {};
        \node at ($(cl)+(2.5cm,0mm)$) (cr) {};
	\node[vertex, label=left:$b$, below of=cl] (bl) {};
	\node[vertex, label=right:$b$, right of=bl, node distance=2.5cm] (br) {};
	\node[radius=0pt, above of=cl] (el) {};
	\draw
	(dl) edge (br)
	(bl) edge (er);
	\draw
	(dl) edge (er);
	\draw 
	(bl) ellipse (6mm and 8mm)
	(br) ellipse (6mm and 8mm)
	 (cl) ellipse (6mm and 8mm) 
	(cr) ellipse (6mm and 8mm) 
	($(cl)!0.5!(bl)$) ellipse (7mm and 30mm)
	($(cr)!0.5!(br)$) ellipse (7mm and 30mm);
	\node [below of=br, node distance=1.1cm] {$B$};
	\node [below of=bl, node distance=1.1cm] {$B$};
	\node [above of=cr, node distance=1.1cm] {$C$};
	\node [above of=cl, node distance=1.1cm] {$C$};
  \end{tikzpicture}

  \end{center}
  \caption{Three stages of the proof that $S_{T;1,2}=A^2\setminus(A\setminus
    C)^2$. Solid lines denote pairs in $S_{T;1,2}$ while the dashed
  line stands for the relation $S_{T;1,2}^\star$.}\label{figS}
\end{figure}

\begin{lemma}\label{lemBinaryblocker}
  If we choose $T$ and $C$ as above, then $S_{T;1,2}=A^2\setminus(A\setminus C)^2$.
\end{lemma}
\begin{proof}
  We begin our proof by showing that $C^2\cap S_{T;1,2}\neq \emptyset$. Assume for a
  contradiction
  that the subuniverse $C^{+S_{T;1,2}}=\{e\in A \colon \exists c\in C,
  (e,c)\in S_{T;1,2}\}$ of $\algA$ is disjoint from $C$. 
	Let $T'=T\times C^{+S_{T;1,2}}$ and consider the corresponding relation
  $S_{T'}$. We claim that $S_{T'}$ is
  a $B$-essential subpower of $\algA$ and that $B^{+S_{T';1,2}}$ is strictly larger than
  $C$. This will contradict the choice of $T$ and $C$.

  Why is $S_{T'}$ a $B$-essential relation? Were there $(\beta,\beta,\dots)\in
  S_{T'}$, we would have
  $(c,\beta,\beta,\dots)\in S_T$ for some $c\in C^{+S_{T;1,2}}$. But then $c$
  would belong to $C$, a contradiction with $C\cap C^{+S_{T;1,2}}=\emptyset$. On the other hand, from $b \in C^{+S_{T;1,2}}$ it follows that  the projection of $S_{T'}$ onto all but one coordinates intersects $B^{\en}$. 

  It remains to show that $C\subsetneq B^{+S_{T';1,2}}$. First of all, if $g\in C$, then
  $(b,g,\beta,\dots,\beta)\in S_{T}$ by Lemma~\ref{lemcinB}. Using $b\in
  C^{+S_{T;1,2}}$ then gives us $g\in B^{+S_{T';1,2}}$, so we have $C\subseteq B^{+S_{T';1,2}}$. 

  To see that the inclusion is proper, we need to take a small detour. Choose
  and fix a $c\in C$. By Lemma~\ref{lemcinB}, we have $(c,b),(b,c)\in S_{T;1,2}$.
 Consider the relation
  $S_{T;1,2}^\star=\{(x,y) \colon (x,y,\alpha,\beta,\beta,\dots)\in S_T\} \leq \algA^2$. 
	Since $B\Jabsorbs \algA$, we have
  $S_{T;1,2}\Jabsorbs \alg{S}_{T;1,2}^\star$ by Lemma~\ref{lemAbsPP}. 
	The symmetry of $S_T$ gives us
  $(b,b,c,\beta,\beta,\dots)\in S_T$
  so $(b,b)\in S_{T;1,2}^\star$. The sequence $c,b,b$ is thus a fence
  of even length from $c$ to $b$ in $S_{T;1,2}^\star$ (see Figure~\ref{figS}
  left). Lemma~\ref{lemFence}
  then implies that there exists a fence of even length from $c$ to $b$ in $S_{T;1,2}$. 
  To reach $b$, this fence cannot use only edges in $C\times C^{+S_{T;1,2}}$, so
	there exist some $d\in C$, $e \in
	C^{+S_{T;1,2}}$ and $f\in A \setminus C$ such that $(d,e),(f,e)\in
	S_{T;1,2}$ (see Figure~\ref{figS} center). Since $S_T$ is symmetric, we
	have $(e,f,\beta,\beta,\dots)\in S_T$, from which it follows that
	$(f,\beta,\dots)\in S_{T'}$. We see that $f\in B^{+S_{T';1,2}}\setminus C$ witnesses
	$C\subsetneq B^{+S_{T';1,2}}$.  

  We have shown that $C\cap C^{+S_{T;1,2}}\neq \emptyset$, that is, there
  exist $d,e\in C$ such that $(d,e)\in S_{T;1,2}$ (see Figure~\ref{figS} right).  
	The set $\{x \in A \colon (d,x) \in S_{T;1,2}\}$ is a subuniverse of $\algA$ (from idempotency) which contains $b$ and $e$. 
	By the second part of Lemma~\ref{lemcinB}, this subuniverse is equal to $A$, in other words, $(d,x)\in S_{T;1,2}$ for every $x\in A$.
  Similarly, $(x,e)\in S_{T;1,2}$ for every $x\in A$. By repeating this argument, we get
   that the set of neighbors of any $f\in C$ is equal to $A$, 
  so $S_{T;1,2}\supseteq A\times C\cup C\times A=A^2\setminus(A\setminus C)^2$.

  It remains to verify that there is no edge $(e,f)\in S_{T;1,2}$ connecting
  vertices $e,f\in
  A\setminus C$.
  If we had such an edge, we would let $T'=T\times \{e\}$ and consider
  $S_{T'}$. Arguments similar to the above give us that $S_{T'}$ is a  $B$-essential relation and 
  $B^{+S_{T';1,2}} \supseteq C\cup \{f\}$, a contradiction with the maximality of $C$.
\end{proof}

To finish the proof of item (5) in Definition~\ref{def-blocker}, we need to generalize $S_{T;1,2}$ to higher arities. 
Denote by $S_{T;1,2,\dots,n}$ the relation
\[
  \{(s_1,\dots,s_n) \colon (s_1,\dots,s_n,\beta,\beta,\dots)\in S_T\} \leq \algA^n.
\]
\begin{lemma}\label{lemGeneralblocker}
  Let $S_T$ be
  chosen as above (ie. $S_T$ is $B$-essential and the size of
  $C=B^{+S_{T;1,2}}$ is maximal). Then
  $S_{T;1,2,\dots,n}=A^n\setminus(A\setminus C)^n$ for all $n \in\en$.
\end{lemma}
\begin{proof}
  Let $X_n$ be the $n$-ary subpower of $\algA$ generated by the tuples (note the extra
  tuple at the end):
\[
  \begin{matrix}
    (a,&b,&b,&\dots,&b,&b),\\
    (b,&a,&b,&\dots,&b,&b),\\
    &&&\ddots,\\
    (b,&b,&b,&\dots,&b,&a),\\
    (b,&b,&b,&\dots,&b,&b).
  \end{matrix}
\]
 We prove by induction on $n$ that for every $n\in\{2,3,\dots\}$ we have $S_{T;1,2,\dots,n}=A^n\setminus
(A\setminus C)^n$ and $X_{n-1}=A^{n-1}$.

The claim holds for $n=2$ by the previous lemma (to see that $X_1=\langle a,b\rangle=A$,
use the minimality of $\algA$ as in the proof of Lemma~\ref{lemcinB}). Assume that the claim holds for some $n$. We show
that it holds for $n+1$.

We first prove $X_n=A^n$. By examining how $R_\infty$ is generated, one sees that
$X_n$ contains the whole $S_{T;1,2,\dots,n}$ which is equal to $A^n\setminus(A\setminus C)^n$ by the induction hypothesis. It
remains to prove that $X_n$ also contains $(A\setminus C)^n$.  Let
$(x_1,\dots,x_{n-1},x_n)$ be a tuple with $x_i\in A\setminus C$ for each $i$.
Since $a\in C$, we know that the tuple $(x_1,\dots,x_{n-1},a)$ is in $X_n$.
Moreover, it is apparent from the generators  
that $X_{n-1}\times \{b\}\subset X_n$, so the induction hypothesis implies $(x_1,\dots,x_{n-1},b)\in X_n$. 
We have shown that the subuniverse $\{z \colon (x_1, \dots, x_n,z) \in X_n\} \leq \algA$ contains $a$ and $b$, therefore it is equal to $A = \langle a,b \rangle$. In particular, $(x_1, \dots, x_n) \in X_n$, as claimed. 

We are now ready to show that $S_{T;1,2,\dots,n+1}=A^{n+1}\setminus(A\setminus
C)^{n+1}$. We begin with the inclusion ``$\supseteq$''.  
Given any $c\in C$, the induction hypothesis gives us that the $n$-ary tuples
$(c,a,b,\dots,b)$ and $(c,b,b,\dots,b)$ lie in $S_{T;1,2,\dots,n}$.
Therefore,
\[
  (c,a,b,\dots,b,\beta,\beta,\dots), (c,b,b,\dots,b,\beta, \beta, \dots )\in S_T.
\]
Using Lemma~\ref{lembbbbb} we get
\[
  (c,a,b,\dots,b,b),(c,b,b,\dots,b,b)\in S_{T;1,2,\dots,n+1}.
\]

Since $S_{T;1,2,\dots,n+1}$ is symmetric, it contains all generators of $C\times
X_n$, therefore  $S_{T;1,2,\dots,n+1}\supseteq C\times X_n=C\times
A^n$. One more use of this symmetry then yields $S_{T;1,2,\dots,n+1}\subseteq A^{n+1}\setminus(A\setminus C)^{n+1}$



It remains to prove that $S_{T;1,2,\dots,n+1}$ does not contain any element
from $(A\setminus C)^{n+1}$. For a contradiction, assume that $(y_1,\dots,y_{n+1})\in (A\setminus
C)^{n+1}\cap S_{T;1,2,\dots,n+1}$. We use a trick similar to the one we employed when proving Lemma~\ref{lemBinaryblocker}:
Let $T'=T\times\{(y_1,\dots,y_n)\}$; we claim that $S_{T'}$ is $B$-essential,
and that $B^{+S_{T';1,2}}\supsetneq C$. This will contradict the maximality of $C$.

The $B$-essentiality of $S_{T'}$ is seen as follows. Since $(y_1,\dots,y_n,y_{n+1})\in
S_{T;1,2,\dots,n+1}$, we have $(y_{n+1},\beta,\beta,\dots)\in S_{T'}$. Moreover, were
$(\beta,\beta,\dots)\in S_{T'}$, we would have $(y_1,\dots,y_{n},\beta,\beta,\dots)\in S_T.$ Therefore, $(y_1,\dots,y_n)\in
S_{T;1,2,\dots,n}\cap (A\setminus C)^n$ which is impossible by the induction hypothesis.

Let now $c\in C$. Then $(y_1,\dots,y_n,c)\in A^n\times C\subset
S_{T;1,2,\dots,n+1}$, which gives us $(c,\beta)\in S_{T';1,2}$. Therefore $C\subset B^{+S_{T';1,2}}$.
Moreover, $y_{n+1}\in B^{+S_{T';1,2}}\setminus C$, making $B^{+S_{T';1,2}}$ strictly
larger than $C$.
\end{proof}

\section{Deciding absorption}
\label{secDeciding}
Armed by Theorem~\ref{MainThm}, we now show how to algorithmically
decide the absorption problem:

\begin{problem}[Deciding absorption]\ \\
{\bf Input:} $\algA$ finite idempotent algebra with finitely many basic operations (the operations are
given by tables of values), $B\leq \algA$.

\noindent
{\bf Output:} ``Yes'' if $B\absorbs \algA$, ``No'' otherwise.
\end{problem}

Note that the size of the
input for our problem is not just $|A|$, but $|A|$ plus the tables of all the
operations on $|A|$, so for one $k$-ary operation, the input size is of the
order of $|A|^k$. We denote this input size $\|\algA\|$.

By Theorem~\ref{MainThm}, it is enough to decide whether $B\Jabsorbs\algA$
and if so, whether there are no $B$-blockers in $\algA$. Of these two
subproblems, the second is the harder one.

Let us begin by devising  a polynomial time algorithm that, given idempotent
$\algA$ and $B\leq \algA$, decides whether $B\Jabsorbs \algA$.  The algorithm is based on the following variant of the result by Ralph Freese and Matt Valeriote~\cite[Proposition~5.7]{matt-complexity-maltsev-conditions}.  

\begin{theorem}\label{thmJPath}
  Let $\algA$ be a finite idempotent algebra and $B\leq \algA$. Then $B\Jabsorbs
  \algA$ if and only if 
  for every $a,c,d\in A$ and every $b_1,b_2\in B$, the digraph $\relstr{G}=(A,E)$ with the
  edge set
  \[
    E=\{(u,v)\colon \exists b\in B, (b,u,v)\in \langle(b_1,a,a), (b_2,c,c),
    (d,a,c)\rangle\}
  \]
contains a directed path from $a$ to $c$.
\end{theorem} 

The subuniverse $\langle(b_1,a,a)$, $(b_2,c,c)$, $(d,a,c)\rangle \leq \algA^3$ 
can be visualised as a colored digraph: an element $(b,u,v)$ is regarded as the edge $(u,v)$ colored by $b$. 
The theorem asks for a $B$-colored directed path from $a$ to $c$.

\begin{proof}
If $B\Jabsorbs\algA$, then applying the J\'onsson absorption terms
$d_0,d_1,\dots,d_n$ to the columns of the matrix
\[
  \begin{pmatrix}
    b_1&a& a\\
    d & a&c\\
    b_2 &c&c\\
  \end{pmatrix}
\]
gives us a sequence of $n+1$ triples whose first coordinates are in $B$ and
the second and third coordinates form a path in $\relstr{G}$ from $a$ to $c$,
proving the ``only if'' implication.

To prove the other, harder, implication, we introduce the property $D(i,j)$, a sort of partial
J\'onsson absorption:
\begin{definition}\label{defDij}
  Denote by $D(i,j)$ the
following property: Whenever 
$\tupl b_1,\tupl b_2 \in B^i$, $\tupl d\in A^i$, $\tupl a,\tupl c\in A^j$, and 
$R$ is the subuniverse of $\algA^{i+2j}$ generated by  $(\tupl b_1,\tupl a,\tupl
a)$, $(\tupl b_2,\tupl c,\tupl c)$,
$(\tupl d,\tupl a,\tupl c)\in R$, then the digraph $\relstr{H}=(A^j,F)$ with the edge set
\[
  F=\{(\tupl u,\tupl v)\in (A^j)^2\colon (\tupl \beta,\tupl
    u,\tupl v)\in R\}
\]
contains a directed path from $\tupl a$ to $\tupl c$. (We use the wild card
$\tupl \beta$ for ``there exists some suitable tuple of elements of $B$''.)
\end{definition}

Observe that $D(1,1)$ is just a reformulation of the ``there exists a path'' condition from the
statement of Theorem~\ref{thmJPath}. We show that $D(1,1)$ implies $D(i,j)$ for
any $i,j\in \en$, which will almost immediately give us a J\'onsson
absorbing chain for $B \Jabsorbs \algA$.

To show that $\algA$ has $D(i,j)$ for any $i,j$, we need three lemmas.
\begin{lemma}\label{lemEdges} 
  Let $\tupl a, \tupl b_1,\tupl b_2, \tupl c, \tupl d$ and $R$ be 
  as in the definition of $D(i,j)$. 
  Then for every $\tupl e\in \langle \tupl a,\tupl c\rangle$, we have
  \[
    (\tupl \beta,\tupl e,\tupl e),\, (\tupl \alpha,\tupl a, \tupl e),\,(\tupl
    \alpha,\tupl  e,\tupl c)\in R.
  \]
\end{lemma} 

\begin{proof} Let $t$ be a term of $\algA$ such that $\tupl e=t(\tupl a,\tuple  
  c)$. We obtain the three tuples we need (in the order
  listed in the statement of the lemma) by applying $t$ to the columns of the three matrices:
  \[
    \begin{pmatrix}
      \tupl b_1&\tupl a&\tupl a\\
      \tupl b_2&\tupl c&\tupl c\\
    \end{pmatrix},\,
    \begin{pmatrix}
 \tupl b_1&\tupl a&\tupl a\\
      \tupl d&\tupl a&\tupl c\\
          \end{pmatrix},\,
    \begin{pmatrix}
 \tupl d&\tupl a&\tupl c\\
      \tupl b_2&\tupl c&\tupl c\\
    \end{pmatrix}.
  \]
\end{proof}

\begin{lemma}\label{lem1st} 
  If $\algA$ satisfies $D(1,j)$ and $D(i,j)$, then $\algA$ satisfies $D(i+1,j)$.  
\end{lemma} 

\begin{proof} 
  Given  $\tupl b_1,\tupl b_2 \in B^{i+1}$, $\tupl d\in
  A^{i+1}$, $\tupl a,\tupl c\in A^j$ and 
	$R = \langle (\tupl b_1,\tuple a, \tuple a), (\tupl b_2, \tupl c, \tuple c), (\tupl d, \tupl a, \tupl c)\rangle \leq\algA^{i+1+2j}$, we construct and examine the digraph $\relstr{H}$ as defined  above. 
	
	Consider the projection of $R$ onto
  all its coordinates except the first one. Using $D(i,j)$, we
  find a sequence $\tupl e_0,\tupl e_1,\tupl e_2,\dots, \tupl
  e_k\in A^j$ and a sequence $f_1,\dots,f_{k}\in A$ such that $\tupl e_0=\tupl
  a$, $\tupl e_k=\tupl   c$, and for each $\ell=1,2,\dots k$ we have
  $(f_\ell\tupl \beta, \tupl e_{\ell-1},\tupl e_{\ell})\in R$ (see
  Figure~\ref{figlem1st}).

  For each $\ell$, the tuples $\tupl e_{\ell-1}$ and $\tupl e_{\ell}$ are in $\langle \tupl a, \tupl c \rangle$, therefore
	some tuples of the form $(\beta\tuple \beta,\tupl e_{\ell-1}, \tuple
	e_{\ell-1})$ and $(\beta\tupl \beta,\tupl e_{\ell},\tupl e_{\ell})$ are
	in $R$ by Lemma~\ref{lemEdges}. Since $(f_\ell\tupl \beta, \tupl
	e_{\ell-1},\tupl e_{\ell})$ is in $R$ as well, by applying $D(1,j)$ to
	the subuniverse $S \leq \algA^{1+2j}$ generated by the coordinates $1,
	i+2, i+3, \dots, i+2j+1$ of the three tuples from this paragraph, we get a $B$-labelled path from $\tupl e_{\ell-1}$ to
  $\tupl e_{\ell}$. Notice that $S$ is a subset of 
	\[
    \{(u,\tupl v,\tupl w)\in A^{1+2j}\colon (u\tupl \beta,\tupl v,\tupl w)\in
    R\},
  \]
	therefore we can lift this path to a $B^{i+1}$-labeled path from $\tupl e_{\ell-1}$
  to $\tupl e_\ell$ in $\relstr{H}$. Concatenating all these paths, we get a path from $\tupl
  a$ to $\tupl c$ in $\relstr{H}$, just as we needed.
\end{proof}

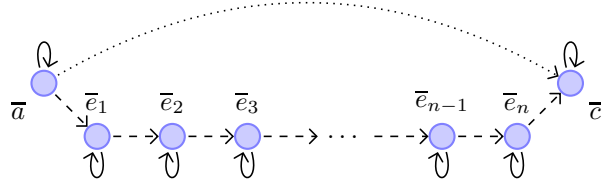
\begin{figure}
  \begin{center}
    \begin{tikzpicture}
    \node[vertex,label=below left:$\tupl a$]  (a) {};
    \node[vertex,below right of=a,label=$\tupl e_1$](e1){};
    \node[vertex,right of=e1,label=$\tupl e_2$](e2){};
    \node[vertex, right of=e2,label=$\tupl e_3$](e3){};
    \node[vertex,right of =a,label=below right:$\tupl c$,node distance=7cm](c){};
    \node[vertex,below left of=c,label=$\tupl e_n$](en){};
    \node[vertex,left of=en,label=$\tupl e_{n-1}$](en1){};
    \node at($(e3)!0.5!(en1)$) (e4){$\dots$};
    \foreach \i in {1,2,3,n1,n} {
      \draw[diedge] (e\i) edge[loop below] (e\i);
    }
    \draw [diedge] 
    (a) edge[loop above](a)
    (c) edge[loop above](c)
    (a) edge[dashed](e1)
    (e1) edge[dashed] (e2)
    (en1) edge[dashed] (en)
    (en) edge[dashed] (c)
    (e2) edge[dashed] (e3)
    (e3) edge[dashed] (e4)
    (e4) edge[dashed] (en1)
    (a) edge[dotted, bend left] (c);
  \end{tikzpicture}
  \end{center}
  \caption{Proving Lemma~\ref{lem1st}. Solid lines represent tuples of $R$ that
    begin with $B^{i+1}$, while dashed lines represent 
  tuples of $R$ that begin with $A\times B^i$. The dotted line is the tuple
$(\tuple d,\tuple a,\tuple c)$.} 
  \label{figlem1st}
\end{figure}

\begin{lemma}\label{lem2nd} 
  If $\algA$ satisfies $D(i,1)$ and $D(i,j)$, then $\algA$ satisfies $D(i,j+1)$.  
\end{lemma} 
\begin{proof}  
  The idea of the proof is similar to that of Lemma~\ref{lem1st}, but we will
  need to use finiteness of $\algA$ a bit more: Consider a relation
  $R\leq\algA^{i+2(j+1)}$ together with $\tupl b_1,\tupl b_2 \in B^i$, $\tupl d \in A^i$,
   and $\tupl a,\tupl c\in A^{j+1}$ as in Definition~\ref{defDij} and
  assume that $R,\tupl a,\tupl c$ form a counterexample to $D(i,j+1)$ such that
  the set $\langle \tupl a,\tupl c\rangle$ is inclusion minimal among all
  possible
  counterexamples.
  
  Take the projection of $R$ to all indices except for the $(i+1)$-st and
  $(i+j+1)$-st, and apply $D(i,j)$.  This yields an ``almost chain'' from
  $\tupl a$ to $\tupl c$: A sequence of tuples $\tupl a=\tupl e_1,\tupl
  f_1,\tupl e_2\dots, \tupl e_k,\tupl f_k=\tupl c$ such that for all
  $\ell=1,\dots,k-1$ we
  have $(\tupl \beta,\tupl f_\ell,\tupl e_{\ell+1})\in R$, and the tuples
  $\tupl e_\ell$ and $\tupl f_\ell$ agree on all coordinates except possibly
  the first one (see Figure~\ref{figlem2nd}).

  We now claim that for each $\ell$ there is a directed path from $\tupl a$ to
  $\tupl f_\ell$ in the digraph $\relstr H=(A^{j+1}, F)$ (see Definition~\ref{defDij}). We
  proceed by induction on $\ell$. 

  For $\ell=1$, we use Lemma~\ref{lemEdges} to realize that
  the tuples $(\tupl \beta,\tupl a,\tupl a)$, $(\tupl \beta, \tupl f_1,\tupl f_1)$, and
  $(\tupl \alpha,\tupl a,\tupl f_1)$ are all in $R$. Then we can 
  take the projection of $R$ onto the coordinates
  $1,2,\dots,i,i+1,i+j+1$, apply $D(i,1)$, and get a path from
  $\tupl a$ to $\tupl f_1$.

  The induction step is quite similar: Assume that there is a path from $\tupl
  a$ to $\tupl f_{\ell-1}$, but no path from $\tupl a$ to $\tupl f_\ell$. We will
  show that then there is a path from $\tupl e_\ell$ to $\tupl f_\ell$, which when
  combined with the path $\tupl a,\dots,\tupl f_{\ell-1},\tupl e_\ell$ gives us a
  path from $\tupl a$ to $\tupl f_\ell$ after all.
  
  Lemma~\ref{lemEdges} shows that the tuples $(\tupl \beta,\tupl a,\tupl a)$,
  $(\tupl \beta, \tupl f_\ell,\tupl f_\ell)$, and   $(\tupl \alpha,\tupl a,\tupl f_\ell)$ are in $R$, so the
  pair $\tupl a,\tupl f_\ell$ witnesses the failure of $D(i,j+1)$.
  Since we assume that $\langle \tupl a,\tupl c\rangle$ is inclusion minimal
  among all counterexamples to $D(i,j+1)$, we must have $\langle \tupl a,\tupl
  f_\ell\rangle =\langle \tupl a,\tupl c\rangle$, so in particular $\tupl e_\ell\in
  \langle \tupl a,\tupl f_\ell\rangle$. From this, a reasoning similar to
  Lemma~\ref{lemEdges} gives us that $(\tupl \alpha,\tupl e_\ell,\tupl
  f_\ell)\in R$. But then we can project $R$ onto the coordinates
  $1,2,\dots,i,i+1,j+i+1$ and use $D(i,1)$ on the tuples 
  $(\tupl \beta, \tupl e_\ell,\tupl e_\ell)$, $(\tupl \beta,\tupl f_\ell,\tupl
  f_\ell)$, and $(\tupl \alpha,\tupl e_\ell,\tupl f_\ell)$
  to get a path from $\tupl f_\ell$ to
  $\tupl e_\ell$, as was needed.

  To conclude the proof that $\algA$ satisfies $D(i,j+1)$, observe that $\tupl f_k=\tupl c$, so we have just
  shown that there is a directed path in $\relstr H$ from $\tupl a$ to $\tupl
  c$.
\end{proof}

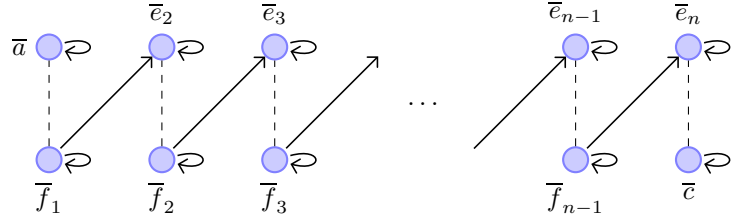
\begin{figure}
  \begin{center}
    \begin{tikzpicture}[node distance=1.5cm]
    \node[vertex,label=left:$\tupl a$]  (e1) {};
    \node[vertex,below of=e1,label=below:$\tupl f_1$](f1){};
    \node[vertex,right of=e1,label=above:$\tupl e_2$](e2){};
    \node[vertex, below of=e2,label=below:$\tupl f_2$](f2){};
    \node[vertex,right of=e2,label=above:$\tupl e_3$](e3){};
    \node[vertex, below of=e3,label=below:$\tupl f_3$](f3){};

    \node[vertex, right of=e3, node distance=4cm,label=above:$\tupl
    e_{n-1}$](en1){};
    \node[vertex,right of=en1,label=above:$\tupl e_n$](en){};
    \node[vertex, below of=en1,label=below:$\tupl f_{n-1}$](fn1){};
    \node[vertex, below of=en,label=below:$\tupl c$](fn){};
    \node at($(f3)!0.5!(en1)$) (dts){$\dots$};

    \node[right of=e3](e4){};
    \node[left of=fn1](fn2){};

    \foreach \i in {1,2,3,n1,n} {
      \draw[diedge] (e\i) edge[loop right] (e\i);
      \draw[diedge] (f\i) edge[loop right] (f\i);
      \draw[dashed] (f\i) edge (e\i);
    }

    \foreach \i/\j in {1/2,2/3,3/4,n2/n1,n1/n} {
      \draw[diedge] (f\i) edge (e\j);
    }
  \end{tikzpicture}
  \end{center}
  \caption{Proving Lemma~\ref{lem2nd}. Solid lines represent tuples of $R$ that
    begin with a tuple from $B^{i}$; dashed lines represent the relation
    ``these two tuples agree on coordinates 2 through $j+2$''. Not pictured:
  Tuples of the form $(\tupl \alpha,\tupl a,\tupl e_k)$.}
  \label{figlem2nd}
\end{figure}

We now repeatedly use Lemmas~\ref{lem1st} and~\ref{lem2nd} until we get that
$\algA$ satisfies $D(|A||B|^2,|A|^2)$. We then consider the three ternary projections
$\pi_1,\pi_2,\pi_3$ as tuples of values; specifically, we evaluate these
projections on all tuples from $A\times B\times A$ (the first block of values), on
all tuples of the form $\{(x,x,y)\colon x,y\in A\}$ (the second block), and on all
tuples of the form $\{(x,y,y)\colon x,y\in A\}$ (the third block). We have three
members of $A^{|A||B|^2+2|A|^2}$. Let $R$ be the relation generated by these
three tuples. Using $D(|A||B|^2,|A|^2)$ on $R$,
we get a path in the digraph $\relstr
H$ that precisely corresponds to a J\'onsson absorption chain for $B \Jabsorbs \algA$.  
This finishes the proof of Theorem~\ref{thmJPath}
\end{proof}

\bigskip

Theorem~\ref{thmJPath} gives us a straightforward algorithm to test whether
$B\Jabsorbs\algA$ for $\algA$ idempotent: Try all possible $a,c,d\in A$,
$b_1,b_2\in B$ (there are $O(|A|^5)$ such choices) and for each choice use the
three triples $(b_1,a,a), (b_2,c,c), (d,a,c)$ to generate a subuniverse
$R\leq \algA^3$. Then go through the list of elements of $R$ and find those that have
their first coordinate in $B$, record their last two entries as edges in 
a new relation $F \subseteq A^2$ and check whether there is a path from $a$ to $c$ in 
the graph $\relstr H=(A,F)$. 

Let us move on to deciding whether $\algA$ contains a $B$-blocker. For this, we make
use of ideas from~\cite{markovic-maroti-mckenzie-blockers}. We start by rewording the
definition of a $B$-blocker:

\begin{lemma}\label{lemBblockers}
  Let $\algA$ be a finite idempotent algebra and $B\leq \algA$. Then $(C,D)$ is a
  $B$-blocker if and only if all the following hold:
\begin{enumerate}
  \item $D \leq \algA$,
  \item $\emptyset\neq C\subset D$,
  \item $C\cap B=\emptyset$,
  \item $D\cap B\neq\emptyset$, and
  \item\label{itmIndices}for each basic operation $t$ of $\algA$, there exists an index $i$ such
    that 
    \[
      t(D,\dots,D,C,D,\dots,D)\subset C,
    \]
    where  $C$ is at the $i$-th coordinate.
\end{enumerate}
\end{lemma}

From Lemma~\ref{lemBblockers} it is apparent that, given
$(C,D)$, we can test in time polynomial in $\|\algA\|$ whether $(C,D)$ is a $B$-blocker. 
Thus, deciding $B$-blockers is in the class
$\compNP$.

Can we do better? Generally no, since deciding whether $\algA$ containts a
$B$-blocker turns out to be \compNP-hard even when it is guaranteed that $\algA$ is idempotent, has directed J\'onsson terms, and $B$ is a singleton (thus, in particular, $B\Jabsorbs \algA$). This is in sharp contrast
with deciding the existence of cube term blockers, which can be done in
$\compP$ when $\algA$ is idempotent~\cite{zhuk-kazda}.

\begin{problem}[Deciding singleton-blockers in idempotent CD algebras]\ \\
{\bf Input:} $\algA$ finite idempotent algebra in a congruence distributive variety with
finitely many basic operations (the operations are given by tables of values),
$b\in A$.

\noindent
{\bf Output:} ``Yes'' if $\algA$ contains a $\{b\}$-blocker, ``No'' otherwise.
\end{problem}

\begin{theorem}\label{thmReduction}
  There is a logarithmic space reduction from 3-SAT to deciding
  singleton-blockers in idempotent CD algebras.
\end{theorem}
  \begin{proof}
    Given a nonempty 3-SAT formula over the
    variables $w_1,w_2,\dots,w_n$ of the form
    \[
      \psi(w_1,\dots,w_n)=\bigwedge_{j=1}^m (\lambda_{j1}\vee \lambda_{j2}\vee
  \lambda_{j3}),
  \]
  where the $\lambda$s stand for literals, we consider the algebra $\algA$ with
  the universe $A=\{0,1,w_1,\dots,w_n\}$, binary basic operations
  $s_1(x,y),\dots,s_n(x,y)$, ternary basic operations
  $t_1(x,y,z),\dots,t_m(x,y,z)$, and ternary basic operations 
  $d_1(x,y,z)$ and $d_2(x,y,z)$ (the last two operations will form a directed J\'onsson chain). 
  
  For
  each $i=1,2,\dots,n$, let
  \[
    s_i(x,y)=\begin{cases}
       x& \text{if $x=y$,} \\
       w_i& \text{if $x=0$, $y=1$,}\\
       1& \text{else.}
    \end{cases}
  \]
  The operations $t_j(x,y,z)$ are defined as follows (note that the rules for $t_j$ depend
  on the shape of the $j$-th clause, so only five rules get used for each particular $t_j$):
  \[
    t_j(x,y,z)=
    \begin{cases}
      x&\text{if $x=y=z$}\\
      w_k&\text{if $\lambda_{j1}=w_k$, $y=z=w_k$, and $x=1$,}\\
       &\text{or $\lambda_{j2}=w_k$, $x=z=w_k$, and $y=1$,}\\
      &\text{or $\lambda_{j3}=w_k$, $x=y=w_k$, and $z=1$.}\\
      0&\text{if $\lambda_{j1}=\neg w_k$, $x=w_k$, and $y=z=0$,}\\
      &\text{or $\lambda_{j2}=\neg w_k$, $y=w_k$, and $x=z=0$,}\\
      &\text{or $\lambda_{j3}=\neg w_k$, $z=w_k$, and $x=y=0$,}\\
      1&\text{else.}
    \end{cases}
  \]
Finally, we define $d_1$ and $d_2$:
 \[
      d_1(x,y,z)=
      \begin{cases}
	x& \text{if $x=y$ or $x=z$}\\
	1& \text{else.}
      \end{cases}
    \]
    \[
      d_2(x,y,z)=
      \begin{cases}
	z& \text{if $y=z$ or $x=z$}\\
	1& \text{else.}
      \end{cases}
    \]
    Observe that $x=d_1(x,x,y)$, $d_1(x,y,y)=d_2(x,x,y)$, $d_2(x,y,y)=y$, and
    $d_1(x,y,x)=d_2(x,y,x)=x$ for all $x,y\in A$, so $x,d_1,d_2,z$ is a chain of directed Jónsson
    terms, so $\algA$ lies in a congruence distributive variety.

    We now claim that $\algA$ has a $\{0\}$-blocker $(C,D)$ if and only if the the formula
  $\psi$ is satisfiable. 

  Assume first that there is a satisfying assignment $e$ for $\psi$. We put
  $D=A$ and make $C$ consist of precisely the element 1 and all variables
  $w_i$ such that $e(w_i)=\True$ (in particular, $0\not\in C$). We show that
  $(C,D)$ satisfies all conditions from Lemma~\ref{lemBblockers}. Of
  these conditions, only item~(\ref{itmIndices}) is not immediately clear.

  For each $i$ and each $c \in C$, we have $s_i(\{c\},A) \subseteq \{1,c\} \subseteq C$, so the blocker
  condition is satisfied for these operations. Similarly, we get
  $d_1(C,A,A)$,
  $d_2(A,A,C)\subset C$ because $C$ contains 1.   The operations $t_j$ are more
  interesting: Given an index $j$, consider the $j$-th clause of $\phi$. Since
  $e$ satisfies $\psi$, one of the literals of the $j$-th clause is true. Without
  loss of generality, let it be the first one. We show that then
  $t_j(C,A,A)\subset C$.

  We need to consider two cases. If $\lambda_{j1}=w_k$ then $e(w_k)=\True$ and
  $w_k\in C$. The possible values of
  $t_j(c,y,z)$ with $c\neq 0$ are $c$, $w_k$ and $1$, all of which lie in $C$ when $c \in C$. If,
  on the other hand, $\lambda_{j1}=\neg w_k$, then $w_k\not\in C$ and the possible values of
  $t_j(c,y,z)$ for $c\not \in\{0,w_k\}$ are $c$ and 1. We see that $(C,D)$ is a blocker.

  Now let $(C,D)$ be a $\{0\}$-blocker. We use $(C,D)$ to get a satisfying
  assignment for $\psi$. First observe that $D=A$: Since $D$ contains 0 and at
  least one other element, it must contain 1 (because $s_i(x,0)=1$ for any
  $x\neq 0$). Then the operations $s_i$ give us that $D=A$ (this is in fact the purpose
  of these operations).

  We now claim that $C$ must also contain 1. If not, then $C$ contains at
  least one variable $w_k$. Consider the expressions
  $t_1(w_k,1,1),t_1(1,w_k,1),t_1(1,1,w_k)$. All of these expressions evaluate to
  1, yet at least one of them must lie in $C$ since $(C,D)$ is a blocker. It
  follows that $1\in C$.

  We see that $C$ has the form $\{1\}\cup V$ for some set of variables $V$.
  Consider the assignment $e$ that makes all variables in $V$ true and the
  rest false. We claim that $e$ satisfies $\psi$. 
	What if we are wrong? Then
  there is a clause, say the $j$-th one, none of whose literals is satisfied by
  $e$. We show that then neither $t_j(C,A,A)$, nor $t_j(A,C,A)$, nor $t_j(A,A,C)$
  lie in $C$. 

  Consider $t_j(C,A,A)$; the other two cases are similar. If $\lambda_{1j}=w_k$, then $w_k\not \in V$, so
  $w_k\not\in C$. Then $t_j(1,w_k,w_k)=w_k\not\in C$. If $\lambda_{1j}=\neg w_k$,
  we have $w_k\in C$ and $t_j(w_k,0,0)=0\not \in C$. This contradiction completes the proof.
\end{proof}

While Theorem~\ref{thmReduction} (together with Theorem~\ref{MainThm} and Theorem~\ref{thmJPath}) implies that deciding absorption is
co-\compNP-complete, this does not need to stop us in practice since the problem of
deciding the existence of a $B$-blocker is fixed parameter tractable when
parametrized by the product of all arities of the basic operations in $\algA$. The idea is to
guess the indices in item~(\ref{itmIndices}) of Lemma~\ref{lemBblockers} in
advance and then check whether there is a $(C,D)$ that would work:

\begin{algorithm}[Deciding $B$-blockers]\label{algBblockers}\ \\
  Input: Idempotent algebra $\algA$ with basic operations $t_1,\dots,t_n$ of arities
  $s_1,\dots, s_n$, a subuniverse $B$ of $\algA$. 
  \begin{enumerate}
    \item For each $c\in A\setminus B$, $b\in B$ and each choice $j_1\in [s_1]$, $j_2\in
      [s_2]$, \dots, $j_n\in[s_n]$, do
      \begin{enumerate}
	\item Let $D=\langle b,c\rangle$,
	\item construct a digraph $\relstr{G}$ with the vertex set $D$ by drawing an edge from
	  $u$ to $v$ whenever $v\in t_i(D,D,\dots,D,u,D,\dots,D)$
	  (here $u$ is at the $j_i$-th coordinate),
	\item let $C$ be the set of all vertices in $\relstr{G}$ to which there is a
	  directed path from $c$ (in other words, $C$ is the smallest set
	  that contains $c$ and together with $D$ satisfies item~(\ref{itmIndices}) of
	  Lemma~\ref{lemBblockers}).
	\item If $C\cap B=\emptyset$, output the blocker $(C,D)$, else iterate.
      \end{enumerate}
    \item If no blocker is found, output ``No.''
  \end{enumerate}
\end{algorithm}
If the algorithm outputs a $(C,D)$, it is easy to verify that this pair 
will satisfy all properties from
Lemma~\ref{lemBblockers}. On the other hand, if $\algA$ has  a
$B$-blocker $(C',D')$, then when the algorithm picks $c\in C'$, $b\in B\cap D'$ and chooses
$j_1,\dots,j_n$ as the indices from item~(\ref{itmIndices}) for $(C',D')$, it
will define $D = \langle b,c \rangle \subseteq D'$ and then construct some $C
\subseteq C'$. It is easy to verify that because $(C',D')$ satisfies the
conditions from Lemma~\ref{lemBblockers}, so does $(C,D)$, and we have a
$B$-blocker.

To analyze the time complexity of Algorithm~\ref{algBblockers} (in the RAM
model of computation),
we can assume that $\algA$ has at least one at least binary basic operation (otherwise we can simply answer ``No.'') and so $|A|^2\leq \|\algA\|$.
Steps (1a) and (1b) each take time $O(\|\algA\|)$, step (1c) is a breadth-first
search that can be done in time $O(|A|^2)$ and step (1d) takes time $O(|A|)$,
so each iteration of Step 1 runs in time $O(\|\algA\|)$.  

There are $O(|A|^2)$ choices of $c$ and $b$ and $\prod_{i=1}^n s_i$ choices of
the indices $j_1,\dots,j_n$, giving us a total time complexity of
$O(\|\algA\||A|^2\cdot \prod_{i=1}^n s_i)$. Therefore, if we keep the number
and arities of basic operations of $\algA$ bounded from above, we get a
polynomial time algorithm that should work well for many idempotent algebras
encountered in the wild.

The following theorem sums up our computational complexity results.

\begin{theorem}
  If $\algA$ is a finite idempotent algebra, $B\leq \algA$ then
  \begin{enumerate}
    \item deciding whether $B\Jabsorbs \algA$ is in \compP,
    \item deciding whether $\algA$ contains a $B$-blocker is \compNP-complete even
      if $\algA$ is in a congruence distributive variety and $B$ is a singleton,
    \item deciding whether $B\absorbs \algA$ is co-\compNP-complete, but is fixed
      parameter tractable if we parametrize the problem by the product of all arities
      of the basic operations in $\algA$.
  \end{enumerate}
\end{theorem}

\section{Conclusion}

We have shown that $B$ is an absorbing subuniverse of a finite idempotent algebra $\algA$ if and only if $B$ J\'onsson absorbs $\algA$ and $\algA$ has no $B$-blockers. Our result can be seen, for finite algebras, as  ``decomposed'' version of the fact~\cite{BIMMVW,MM08b} that a finite algebra $\algA$ has an NU term if and only if $\algA$ has J\'onsson terms and a cube term. Indeed, in view of~Proposition~\ref{abs-and-nu}, Corollary~\ref{cor-Jonsson} and \cite[Theorem~3.4]{markovic-maroti-mckenzie-blockers}, this fact can be equivalently phrased as follows:
Finite idempotent algebra $\algA$ has all singletons absorbing if and only if $\algA$ has all singletons J\'onsson absorbing and has no $\{a\}$-blockers for any $a \in A$.  (For non-idempotent finite algebras, we can formulate a similar statement by replacing $\algA$ with the full idempotent reduct of $\algA$.) While having all singletons absorbing (that is, having an NU term) is quite a strong property of algebras, having some proper absorbing subuniverse is quite common -- for instance, any pair of non-commuting congruences of a finite idempotent algebra $\algA$ with a Taylor term forces a proper absorbing subuniverse in a subalgebra of $\algA$ (by the Absorption Theorem~\cite{barto-kozik-cyclic-terms-and-csp}). For this reason, it seems that our theorem promises much wider applicability. 

Several other results about algebras with NU terms or J\'onsson terms can be decomposed in a similar way. For instance,
Proposition~\ref{propEssentialRel} can be regarded as a decomposed version of a part of the Baker-Pixley theorem~\cite{BP75}, and Theorem~\ref{thmJPath} generalizes \cite[Proposition~5.7]{matt-complexity-maltsev-conditions} in a similar way. A natural questions arise -- how to generalize, in a such a way, other results about algebras with NU terms or J\'onsson terms (eg. J\'onsson's Lemma) and what other Maltsev conditions have useful decomposed versions? 

As an application of the main theorem we have shown that absorption $B \absorbs \algA$ is decidable, where $\algA$ is given by tables of operations. Another way how to define an algebra is by means of relations -- to every relational structure $\relstr{A}$ one assigns so called \emph{algebra of polymorphism} whose operations are those which are compatible with every relation of $\relstr{A}$. A version of the main theorem for polymorphism algebras and decidability of the corresponding absorption problem is provided in~\cite{barto-bulin-absorption}. 

\section*{Acknowledgements}
Libor Barto and Alexandr Kazda were supported by the the Grant Agency of the Czech Republic, grant GA\v CR
13-01832S. Alexandr Kazda was supported by European Research Council under the
European Unions Seventh Framework Programme (FP7/2007-2013)/ERC grant agreement
no 616160.

\bibliographystyle{plain}
\bibliography{citations}

\end{document}